\documentclass[a4paper,12pt]{amsart}
\usepackage{amsfonts}
\usepackage{amsmath,array}
\usepackage{amssymb}
\usepackage{mathrsfs}
\usepackage[colorlinks]{hyperref}
 \usepackage{graphicx}
\usepackage{enumerate}
\usepackage[usenames]{color}
\usepackage{cases}
\newtheorem{thm}{Theorem}[section]
\newtheorem{cor}[thm]{Corollary}
\newtheorem{lem}[thm]{Lemma}
\newtheorem{prop}[thm]{Proposition}

\newtheorem{cond}[thm]{Condition}
\numberwithin{equation}{section}
\usepackage[english]{babel} 
\usepackage{blindtext}

\theoremstyle{definition}
\newtheorem{definition}[thm]{Definition}
\newtheorem{rem}[thm]{Remark}

\begin{document}

\title[Sobolev embeddings]{Fractional Sobolev embeddings on noncommutative torus}

\author[F. Sukochev]{F. Sukochev}
\address{
F. Sukochev:
  \endgraf
School of Mathematics and Statistics, University of New South
Wales,  
  \endgraf
 Kensington, 2052, NSW,
 \endgraf
  Australia  
  \endgraf
   {\it E-mail address} {\rm f.sukochev@unsw.edu.au}
  }

\author[R. Tastankul]{R. Tastankul}
\address{
 R. Tastankul:
  \endgraf
  Institute of Mathematics and Mathematical Modeling, 050010, Almaty, 
  \endgraf
  Kazakhstan
  \endgraf
  {\it E-mail address} {\rm ramazan.tastankul@mail.ru} 
  }

 \author[K. Tulenov]{K. Tulenov}
\address{
  K. Tulenov:
  \endgraf
  Institute of Mathematics and Mathematical Modeling, 050010, Almaty, 
  \endgraf
  Kazakhstan 
  \endgraf
  and 
   \endgraf
School of Mathematics and Statistics, University of New South
Wales,  
  \endgraf
 Kensington, 2052, NSW,
 \endgraf
  Australia 
  \endgraf
  {\it E-mail address} {\rm tulenov@math.kz} 
  }
\author[D. Zanin]{D. Zanin}
\address{
D. Zanin:
  \endgraf
School of Mathematics and Statistics, Central South University
  \endgraf
 Changsha, 410075, 
 \endgraf
  China  
  \endgraf
   {\it E-mail address} {\rm d.zanin@csu.edu.cn }
  } 

\date{}

\begin{abstract}

In this paper, we study the noncommutative fractional symmetric Sobolev spaces on noncommutative torus. We prove noncommutative distributional fractional Sobolev inequality and as its application, we obtain Sobolev embeddings. In order to obtain these results, we first prove a noncommutative version of the famous O'Neil inequality for the convolution. As a first application of our main results, we obtain a Cwikel–Solomyak–type estimate. As an another application, we show a $L_2$-time decay for the mild solution of the Cauchy problem for the diffusion equation in this noncommutative setting. When $\theta=0$, our results recover many known results on Sobolev embedding on the torus, and a classical result on Cauchy problem.
\end{abstract}



\maketitle

\tableofcontents

\section{Introduction}

Sobolev embedding theorems are among the most powerful and widely used tools in modern analysis, particularly in the study of partial differential equations (PDEs). Introduced by S.~L.~Sobolev in the 1930s~\cite{Sobolev}, these results establish a bridge between weak derivatives and classical function spaces. 

More precisely, let $d \in \mathbb{N}$, $1 < p < \infty$ and $0 < s < \frac{d}{p}$. If the exponents $p$ and $q$ are related by $\frac{1}{p} - \frac{1}{q} = \frac{s}{d},$ then there is a continuous embedding $ W^{s}_{p}(\mathbb{R}^{d}) \hookrightarrow L_{q}(\mathbb{R}^{d}),$ which means that every $u \in W^{s}_{p}(\mathbb{R}^{d})$ satisfies the estimate $\|u\|_{L_{q}(\mathbb{R}^{d})} \leq c \|u\|_{W^{s}_{p}(\mathbb{R}^{d})},$ where the constant $c>0$ is independent of $u$ and depends only on $d$, $p$ and $s$.

On bounded domains, an analogous embedding result is available under natural geometric assumptions (see~\cite[Part II, Chapter 5, Theorem 2]{Evans}). If $\Omega \subset \mathbb{R}^{d}$ is a bounded domain with a $C^{1}$ boundary, $1 \leq p < d$, and $s < \frac{d}{p}$, then the Sobolev space $W^{s}_{p}(\Omega)$ is compactly embedded into $L_{q}(\Omega)$, where $q = \frac{dp}{d - sp}.$ In other words, the embedding $ W^{s}_{p}(\Omega) \hookrightarrow L_{q}(\Omega)$ is continuous and compact.

When the integrability exponent exceeds the dimension, that is, when $d < p \leq \infty$, Sobolev embeddings yield not only improved integrability but also H\"older regularity (\cite[Part II, Chapter 5, Theorem 4]{Evans}). This phenomenon is expressed by Morrey’s inequality, which asserts that for $u \in C^{1}(\mathbb{R}^{d})$ one has $\|u\|_{C^{0,\gamma}(\mathbb{R}^{d})} \leq C \|u\|_{W_{p}^{1}(\mathbb{R}^{d})},
\ \gamma = 1 - \frac{d}{p},$ where the constant $C>0$ depends only on $p$ and $d$. In particular, for $p>d$ the space $W^{1}_{p}(\mathbb{R}^{d})$ continuously embeds into the H\"older space $C^{0,\gamma}(\mathbb{R}^{d})$ (\cite[Part II, Chapter 5]{Evans}) with some exponent $\gamma \in (0,1)$.

Trudinger~\cite{T1} achieved a critical milestone in 1967, and precisely characterized the limiting case $p = d$. He proved that for a domain $\Omega$ satisfying a cone condition, the space $W^{1}_{d}(\Omega)$ is continuously embedded into the Orlicz space $L_{\Phi}(\Omega)$ associated with the Young function
\begin{equation*}
    \Phi(t) = e^{|t|^{\frac{d}{d-1}}} - 1.
\end{equation*}
This result is sharp, meaning the exponent $\frac{d}{d-1}$ in the exponential cannot be increased. In their subsequent work~\cite{DT}, Donaldson and Trudinger provided a comprehensive and elegant generalization of the entire Sobolev embedding theory to the scale of Orlicz spaces. They systematically studied spaces $W^{1}_{L_{\Psi}}(\Omega)$, where the $L_{p}$ condition is replaced by an Orlicz space defined by a $N$-function $\Psi$. Their central result is a unified embedding theorem that recovers all the classical results as special cases.

In~\cite{Klimov1969}, Klimov further generalized these embedding theorems to the setting of symmetric (rearrangement invariant) function spaces (which encompass $L_p$, Orlicz, Lorentz, and Marcinkiewicz spaces). Let $\Omega \subset \mathbb{R}^{d}$ be a bounded domain with Lebesgue measure equals to $m(\Omega)$. If we assume $m(\Omega)=1,$ then a central insight of his work is that the embedding properties of the Sobolev spaces $W^{k}_{E}(\Omega)$ are equivalent to the boundedness properties of specific one-dimensional integral operators  (see also~\cite{Klimov1970, Klimov1972}) defined on $E(0,1)$ for $d\geq3:$  
\begin{equation}\label{OperatorA}
\begin{split}
    (A f)(t) = \int\limits_{t}^{1} s^{\frac{1}{d}-1} f(s)  ds, \;(A_k f)(t) = \int\limits_{t}^{1} s^{\frac{k}{d}-1} (C\mu(f))(s) ds, \,\ 2\leq k\leq d-1,\, t\in(0,1),
\end{split}
\end{equation}
where $(C\mu(f))(s) =\frac{1}{s}\int_{0}^s\mu(s,f)d\xi$ and $\mu(f)$ is the decreasing rearrangement of $f.$
Specifically, Klimov established that the Sobolev function space $W^{1}_{E}(\Omega)$ (that is the space of distributions on $\Omega$ whose (weak) derivatives belong to the symmetric function space $E(\Omega)$) embeds into a symmetric function space $F(\Omega)$ if and only if the operator $A:E(0,1)\to F(0,1)$ is bounded. Similar results were obtained for the Sobolev spaces $W^{k}_{E}(\Omega)$ by Cianchi, Edmunds, Kerman, and Pick in~\cite{EKP} (see also~\cite{CPS, Cianchi, GO}) in terms of the operator $A_k.$

Recently, in \cite{SZ}, two of the present authors established a distributional version of the Sobolev inequality: for $d\in\mathbb{N}$ and for any $f\in L_2(\mathbb{R}^d)$,
\begin{equation}\label{distributional sobolev}  
\mu\big((1-\Delta_{\mathbb{R}^{d}})^{-\frac{d}{4}}f\big)\chi_{(0,1)}\prec\prec c_{d}\,T(\mu(f)\chi_{(0,1)}),
\end{equation}
where $\mu(f)$ denotes the decreasing rearrangement of $f$, $\prec\prec$  stands for Hardy--Littlewood--Polya submajorization and $T$  defines the following Hardy type operator $T:L_{p}(0,1)\to L_{p}(0,1), 1\leq p\leq\infty$:
\begin{equation}\label{T def}
 (Tf)(t) = \frac{1}{t^{\frac{1}{2}}} \int\limits_0^t f(s) \, ds + \int\limits_t^1 \frac{f(s)}{s^{\frac{1}{2}}} \, ds, \ \ f \in L_{p}(0,1),
\end{equation}
which is the main technical tool in proving all the mentioned results in \cite{SZ}.

Let $\Lambda_{\psi}^{(2)}$ be the 2-convexification of the Lorentz space $\Lambda_{\psi}$ with 
\begin{equation*}
    \psi(t)=\begin{cases}
        \dfrac{1}{\log(\frac{e}{t})}, &t\in(0,1)\\
        t, &t\in [1,\infty).
    \end{cases}
\end{equation*}
Since $T:L_2(0,1)\to\Lambda_{\psi}^{(2)}(0,1)$ (see Lemma 9 in \cite{SZ}), it follows from the distributional estimate \eqref{distributional sobolev} that
$$(1-\Delta_{\mathbb{R}^{d}})^{-\frac{d}{4}}:L_2(\mathbb{R}^d)\to \Lambda_{\psi}^{(2)}(\mathbb{R}^{d}).$$
Since the image of the latter mapping is exactly $H^{\frac{d}{2}}_{2}(\mathbb{R}^{d}),$ it follows that
\begin{equation*}
H^{\frac{d}{2}}_{2}(\mathbb{R}^{d}) \hookrightarrow \Lambda_{\psi}^{(2)}(\mathbb{R}^{d}).
\end{equation*}
Moreover, the latter embedding is optimal in the following sense. 
\begin{thm}\cite[Corollary 14]{SZ}\label{TheoremIntro1} Let $d\geq1$ and let $E(\mathbb{R}^{d})$ be a symmetric Banach function space on $\mathbb{R}^{d}$. If
\begin{equation*}
    (1-\Delta_{\mathbb{R}^{d}})^{-\frac{d}{4}}: L_{2}(\mathbb{R}^{d})\to E(\mathbb{R}^{d}),
\end{equation*}
then $\Lambda_{\psi}^{(2)}(\mathbb{R}^{d})\subset (E+L_{\infty})(\mathbb{R}^{d}).$ 
\end{thm}

This theorem, in turn, enhanced earlier results due to Hansson \cite{Hansson}, Brezis and Wainger \cite{BW}, and Cwikel and Pustylnik \cite{CP}. The above mentioned papers  \cite{CPS, Cianchi, EKP, Klimov1969, Klimov1970, Klimov1972} studied Sobolev spaces defined via $k$-th order weak (distributional) derivatives.

The operator $A$ in \eqref{OperatorA} was introduced in Klimov’s work in the \emph{first-order} case ($k=1$), corresponding to the Sobolev space $W^{1}_{E}$ defined via first-order weak (distributional) derivatives. In contrast, in the present paper we work with \emph{fractional} Sobolev spaces $H_{\mathcal{E}}^s$ defined by the Bessel potential $(1-\Delta_{\theta})^{\frac{s}{2}}, \,\ s\in\mathbb{R}$, where $\Delta_{\theta}$ is the Laplacian on noncommutative torus defined below in \eqref{NCLaplacian}, and the central object throughout is the Hardy-type operator $T$ defined in \eqref{T def}, which governs the distributional fractional Sobolev inequalities that we establish, both in the commutative (see, \cite{SZ}) and in the noncommutative (quantum torus) setting. Even in the first-order setting, we work with the operator $T$ rather than the operator $A$.
In the first-order case with $d=2$, the operator $A$ coincides with the second (tail) term in the definition of the operator $T$. For \emph{higher-order} derivatives ($k\ge 2$), the relevant Sobolev spaces are $W^{k}_{E}$, defined in terms of weak derivatives up to order $k$, and the associated operator in the framework of Klimov and subsequent authors is $A_k$, given in \eqref{OperatorA}. In our framework, however, we continue to work with the operator $T$. Moreover, in the special case $k=\frac{d}{2}$, the operators $T$ and $A_k$ are equivalent (see \cite[Theorem 3.1]{EKP}).

In this work, we employ and enhance the approach of \cite{SZ}. Firstly, we obtain a noncommutative version of the famous O'Neil inequality (Corollary \ref{convolution-cor}) for convolution mapping, and, we establish a noncommutative distributional fractional Sobolev inequality (Theorem~\ref{TheoremMain}), formulated in terms of Hardy–Littlewood–Pólya submajorization with the operator $T$. As a result, we obtain a noncommutative version (Corollary~\ref{ThmSEmb}) of the Sobolev inequality \cite[Proposition~7]{SZ}. 

In recent years, the development of noncommutative analysis has inspired many contributions to noncommutative versions of Sobolev embeddings~\cite{L, McDSX, Spera, XXY}. An early contribution to the extension of Sobolev theory to noncommutative algebras was made by Spera \cite{Spera}. Building on the initial ideas introduced in \cite{Spera}, a comprehensive theory of function spaces on noncommutative torus $\mathbb{T}^{d}_{\theta}$  was subsequently developed in \cite{XXY}. This framework encompasses fractional (or potential) Sobolev spaces and their generalizations, including Besov spaces (see also \cite{RST}) and Triebel–Lizorkin spaces. In particular, \cite{XXY} establishes a noncommutative analog of the classical Sobolev embedding theorem on $\mathbb{T}^{d}_{\theta}$ (see \cite[Theorem~6.6]{XXY}): if $d\geq 2$ and $1<p<q<\infty$ satisfy $ \frac{1}{q} = \frac{1}{p} - \frac{k}{d},$ then the following continuous embedding holds: $W^{k}_{p}(\mathbb{T}^{d}_{\theta}) \hookrightarrow L_{q}(\mathbb{T}^{d}_{\theta}).$ In a related direction, Lafleche \cite{L} investigated noncommutative analogues of these constructions within the context of quantum mechanics and semiclassical analysis.

As an application of our main theorem, we study noncommutative fractional symmetric Sobolev spaces. In particular, by using connections with Hardy type operator $T$, we obtain a sufficient condition for noncommutative Sobolev embeddings (Theorem \ref{SymSobolevTh}). In particular, when $\theta=0$, our result in Theorem \ref{SymSobolevTh} coincides with the optimal embedding Theorem \ref{TheoremIntro1} and extends all classical results on the torus due to Hansson \cite{Hansson}, Brezis and Wainger \cite{BW}, and Cwikel and Pustylnik \cite{CP}.

Namely, we obtain that if $T$ is bounded from a symmetric space $E(0,1)$ to another symmetric space $F(0,1),$ then we have the following embedding
$$H^{\frac{d}{2}}_{\mathcal{E}}(\mathbb{T}^{d}_{\theta}) \hookrightarrow \mathcal{F}(\mathbb{T}^{d}_{\theta}),$$
where $H^{\frac{d}{2}}_{\mathcal{E}}(\mathbb{T}^{d}_{\theta})$ is the fractional Sobolev
space on the noncommutative torus $\mathbb{T}^{d}_{\theta}$ defined in
Subsection~\ref{NCSobolevSpace}, and $\mathcal{F}(\mathbb{T}^{d}_{\theta})$ is
a symmetric operator spaces corresponding to the symmetric function space $F(0,1)$ (see, Subsection \ref{SubsNCOpSp}).

\subsection{Applications of fractional Sobolev embedding}
We present two applications of our main result.
The first application is a Cwikel–Solomyak–type estimate. In \cite{SZ}, the authors proved the following result as an application of Theorem~\ref{TheoremIntro1}.
\begin{thm}\cite[Proposition 19]{SZ}\label{TheoremIntro2} Let $d\in\mathbb{N}$ and let $\mathtt{M}_{\psi}(\mathbb T^d)$ be the Marcinkiewicz space associated with function $\psi$. Let $f\in \mathtt{M}_{\psi}(\mathbb{T}^{d})$. We have
    \begin{equation*}
        \|(1-\Delta_{\mathbb{T}^{d}})^{-\frac{d}{4}}M_{f}(1-\Delta_{\mathbb{T}^{d}})^{-\frac{d}{4}}\|_{L_{2}(\mathbb{T}^{d})\to L_{2}(\mathbb{T}^{d})}\leq c_{d}\|f\|_{\mathtt{M}_{\psi}(\mathbb{T}^{d})},
    \end{equation*}
where $\Delta_{\mathbb{T}^{d}}$ is the Laplacian on the classical $d$-torus $\mathbb{T}^{d}$ and $M_{f}$ denotes a multiplication operator by $f$ on $L_{2}(\mathbb{T}^{d})$.
\end{thm}

Motivated by this, we use the fractional Sobolev embedding established in this paper to obtain a noncommutative analogue of this optimal estimate. Specifically, as an application of Theorem~\ref{TheoremMain}, we prove a noncommutative Cwikel–Solomyak–type estimate (Theorem~\ref{C-SforM}), which in the commutative case $\theta=0$ reduces exactly to Theorem~\ref{TheoremIntro2}.

As an another application of our noncommutative Sobolev inequality (Corollary~\ref{ThmSEmb}), we obtain the $L_2$-time decay estimate of the heat-type equation in noncommutative setting (Theorem \ref{L2timeest}). The decay properties of solutions to space-fractional diffusion equations have been thoroughly investigated in the commutative case ($\theta = 0$). Early studies on space-fractional heat diffusion were conducted in \cite{Chasseigne2006} and \cite{IgnatRossi2010}, with a more comprehensive analysis for the case of the classical derivative provided in \cite{Reference38}. The asymptotic behavior and decay of solutions were further analyzed in \cite{Vazquez2007}. For the specific endpoint cases of fractional derivative, related results were established in \cite{VergaraZacher2010} and later extended in collaboration with the present authors in \cite{VergaraPresentAuthors} (see also \cite{KSZ}, \cite{STo} and references therein). 
Recently, similar problems were studied in noncommutative setting in \cite[Theorem 4.1.2]{STT}, where the authors  established $L_{p}$-$L_{q}$ decay estimates for the range $1<p\leq 2\leq q<\infty$.  When one considers the case $p=q=2$, their method no longer yields an $L_{2}$-time decay estimate for $t>0$: the argument loses the dependence on $t$, resulting only in a uniform (time--independent) bound. In this paper, we employ a different approach, mainly using Theorem \ref{TheoremMain}, and complement the main results of \cite{STT} for the case $p=q=2$ (see Theorem \ref{L2timeest}), obtaining an $L_2$-time decay estimate of the solution.

\section{Preliminaries}

\subsection{Symmetric (quasi-)Banach function spaces}

Let $\Omega\subset\mathbb{R}^{d}, d\geq1$ be a bounded domain and $m$ be the Lebesgue measure on $\Omega$, and let $S(\Omega)$ denote the space of all $m$-measurable real-valued functions on $\Omega$. For a function $f \in S(\Omega)$, its decreasing rearrangement $\mu(f)$ is defined by
$$
\mu(t,f) = \inf\{ s \geq 0 : m(\{|f| > s\}) \leq t \}, \ \ t \geq 0.
$$

For $1 \leq p \leq \infty$, let $L_{p}(\Omega)$ represents the classical Lebesgue space of $p$-integrable functions (or essentially bounded functions in the case $p=\infty$). The Lorentz space $L_{p,q}(\Omega), \ 1 \leq p, q \leq \infty$, consists of all $f\in S(\Omega)$ for which the following quasi-norm is finite:
\begin{equation*}
    \|f\|_{L_{p,q}(\Omega)}=\left\{ \begin{array}{rcl}
         \left(\int\limits_{0}^{m(\Omega)}\left(t^{\frac{1}{p}}\mu(t,f)\right)^q\frac{dt}{t}\right)^{\frac{1}{q}}, & \mbox{for}
         & q<\infty, \\ \sup\limits_{t>0}t^\frac{1}{p}\mu(t,f),\;\;\;\;\;\;\;\;\;\;\;\; & \mbox{for} & q=\infty. 
                \end{array}\right. 
\end{equation*}
For the case $q=\infty$, we also define the following equivalent quasi-norm:
\begin{equation}\label{LorentzEqvNorm}
    \|f\|_{L_{p,\infty}(\Omega)}=\sup\limits_{0<s<m(\Omega)} s \, m(\{t\in\Omega: |f(t)|>s\})^{\frac{1}{p}}=\sup\limits_{0<t<m(\Omega)}t^\frac{1}{p}\mu(t,f), \ \ 1\leq p<\infty. 
\end{equation}
For a comprehensive treatment of these spaces, we refer the reader to \cite{G2014}.

\begin{definition}\label{Sym}
A space $(E(\Omega),\|\cdot\|_{E(\Omega)})$ is called a \emph{symmetric (quasi-)Banach function space} on $\Omega$ if it satisfies the following properties:
\begin{enumerate}[{\rm (a)}]
    \item $E(\Omega)$ is a linear subspace of $S(\Omega)$;
    \item $(E(\Omega),\|\cdot\|_{E(\Omega)})$ forms a complete (quasi-)normed space;
    \item For any $f \in E(\Omega)$ and $g \in S(\Omega)$ with $\mu(g) \leq \mu(f)$, it follows that $g \in E(\Omega)$ and $\|g\|_{E(\Omega)} \leq \|f\|_{E(\Omega)}$.
\end{enumerate}
\end{definition}

We say that a function $g\in S(\Omega)$ is \emph{submajorized} by $f$ in the sense of Hardy-Littlewood-Pólya (denoted $\mu(g) \prec\prec \mu(f)$) if
$$
\int\limits_{0}^t \mu(s,g) \, ds \leq \int\limits_{0}^t \mu(s,f) \, ds, \ t \geq 0.
$$

\begin{definition}
A (quasi-)Banach function space $E(\Omega) \subset S(\Omega)$ is termed \emph{fully symmetric} if whenever $f \in E(\Omega)$ and $g \in S(\Omega)$ satisfy $\mu(g) \prec\prec \mu(f)$, it follows that $g \in E(\Omega)$ and $\|g\|_{E(\Omega)} \leq \|f\|_{E(\Omega)}$.
\end{definition}

Let $\psi$ be a given increasing, concave and continuous function on $(0,\infty)$ (or $(0,1)$) with $\psi(0)=0$. The Lorentz space $\Lambda_{\psi}(\Omega)$ associated with function $\psi$ is defined as follows
\begin{equation*}
    \Lambda_{\psi}(\Omega)=\left\{ f\in S (\Omega) : \int\limits_{0}^{m(\Omega)}\mu(s,f)\,d\psi(s)<\infty \right\},
\end{equation*}
and the Marcinkiewicz space $\mathtt{M}_{\psi}(\Omega)$ associated with function $\psi$ is defined by setting
\begin{equation*}
    \mathtt{M}_{\psi}(\Omega)=\left\{ f\in S (\Omega) :\sup_{0<t<1}\frac{1}{\psi(t)} \int\limits_{0}^{t}\mu(s,f)\,ds<\infty \right\}.
\end{equation*}
 The $L_{p}$-spaces, Lorentz $L_{p,q}$ spaces, Lorentz and  Marcinkiewicz spaces are classical examples of symmetric function spaces \cite{BS, KPS, LT}, moreover,  $L_{p}, 1\leq p\leq\infty,$ $L_{p,q}, 1<p<\infty, 1\leq q \leq\infty$, the Lorentz space $\Lambda_{\psi}$ and the Marcinkiewicz space $\mathtt{M}_{\psi}$ are fully symmetric spaces (\cite[Chapter 6]{DdPS}).

By $E^{(2)}(\Omega)$ we define 2-convexification of a given symmetric space $E(\Omega)$ (see, for example \cite{DdPS, LT}):
\begin{equation*}
   E^{(2)}(\Omega)=\left\{ f\in S (\Omega) :|f|^{2}\in E(\Omega)\right\},
\end{equation*}
\begin{equation}\label{2ConvNorm}
     \|f\|_{E^{(2)}(\Omega)}=\||f|^{2}\|_{E(\Omega)}^{\frac{1}{2}}, \ \ f\in E(\Omega).
\end{equation}
It should be pointed out that, if $E$ is fully symmetric, then $E^{(2)}$ is fully symmetric (\cite[Proposition 6.6.9]{DdPS}).
\subsection{Hardy type operator}

Now, we define Hardy-type operator, which plays a key role in the proof of our main results.
Let $f\in L_{p}(0,1), \ 1\leq p\leq\infty$, then we define following Hardy-type operator $T:L_{p}(0,1)\to L_{p}(0,1)$ (see, \cite[Chapter II, p. 168]{KPS} and \cite{SZ})
\begin{equation}\label{operatorT}
    (Tf)(t) = \frac{1}{t^{\frac{1}{2}}} \int\limits_0^t f(s) \, ds + \int\limits_t^1 \frac{f(s)}{s^{\frac{1}{2}}} \, ds, \ \ f \in L_p(0,1), \ 1\leq p\leq\infty .
\end{equation}
When $p=2$, the operator $T:L_{2}(0,1)\to L_{2}(0,1)$ is a Hilbert-Schmidt operator (for more details on $T$, see \cite[p. 21]{SZ}).

Let us define the following concave function
\begin{equation}\label{PhiPsiFunc}
\phi(t)=\frac{1}{\log(\frac{e}{t})}, \ \ t\in(0,1).
\end{equation}

By $\Lambda_{\phi}(0,1)$ we denote the Lorentz space associated with the function $\phi$.

\begin{prop}\label{ThmTBdd} Let $T$ be the operator defined by \eqref{operatorT} and let $\phi$ be the function defined by \eqref{PhiPsiFunc}. Then we have
    
 \begin{itemize}
             \item [(i)] $T$ is bounded on $L_{p}(0,1), \, 1\leq p \leq\infty$;
                       \item [(ii)]  $T$ is bounded from $L_{2}(0,1)$ to $\Lambda^{(2)}_{\phi}(0,1).$
                     \end{itemize}
\end{prop}
\begin{proof} (ii) was already obtained in \cite[Lemma 9]{SZ}. Let us prove part (i). Indeed, let $f\in L_{1}(0,1)$ and let us calculate $L_{1}$-norm of $Tf$. By changing the order of integration and applying H\"older inequality, for the first term of \eqref{operatorT} we get
\begin{equation*}
    \begin{split}
    \left\|\frac{1}{t^{\frac{1}{2}}}\int\limits_{0}^{t}f(s)\,ds\right\|_{L_{1}(0,1)} & \leq \int\limits_{0}^{1}\frac{1}{t^{\frac{1}{2}}}\int\limits_{0}^{t}|f(s)|\,dsdt= \int\limits_{0}^{1}|f(s)|\int\limits_{s}^{1}\frac{1}{t^{\frac{1}{2}}}\,dtds\\
    &=2\int\limits_{0}^{1}|f(s)|(1-s^{\frac{1}{2}})\,ds\leq 2 \sup\limits_{0< s< 1}(1-s^{\frac{1}{2}})\|f\|_{L_{1}(0,1)}=2\|f\|_{L_{1}(0,1)},
    \end{split}
\end{equation*}
and for the second term
\begin{equation*}
    \begin{split}
    \left\|\int\limits_{t}^{1}\frac{f(s)}{s^{\frac{1}{2}}}\,ds\right\|_{L_{1}(0,1)} & \leq \int\limits_{0}^{1}\int\limits_{t}^{1}\frac{|f(s)|}{s^{\frac{1}{2}}}\,dsdt= \int\limits_{0}^{1}|f(s)|\int\limits_{0}^{s}\frac{1}{t^{\frac{1}{2}}}\,dtds\\
    &=2\int\limits_{0}^{1}|f(s)|s^{\frac{1}{2}}\,ds\leq 2 \sup\limits_{0< s< 1}s^{\frac{1}{2}}\|f\|_{L_{1}(0,1)}=2\|f\|_{L_{1}(0,1)}.
    \end{split}
\end{equation*}
Hence,
\begin{equation*}
    \|Tf\|_{L_{1}(0,1)}\leq4\|f\|_{L_{1}(0,1)}, \ \ f\in L_{1}(0,1).
\end{equation*}

Now, let $f\in L_{\infty}(0,1).$ Then we have
\begin{equation*}
\begin{split}
    \|Tf\|_{L_{\infty}(0,1)}&=\sup\limits_{0< t< 1}|Tf(t)|\leq\sup\limits_{0< t< 1} \frac{1}{t^{\frac{1}{2}}}\int\limits_{0}^{t}|f(s)|\,ds+\sup\limits_{0< t< 1} \int\limits_{t}^{1}\frac{|f(s)|}{s^{\frac{1}{2}}}\,ds\\
    &\leq \|f\|_{L_{\infty}(0,1)}\sup\limits_{0< t< 1}t^{\frac{1}{2}}+2\|f\|_{L_{\infty}(0,1)}\sup\limits_{0< t< 1}(1-t^{\frac{1}{2}})= 3\|f\|_{L_{\infty}(0,1)},
\end{split}
\end{equation*}
and we obtain
\begin{equation*}
    \|Tf\|_{L_{\infty}(0,1)}\leq 3\|f\|_{L_{\infty}(0,1)}, \ \ f\in L_{\infty}(0,1).
\end{equation*}
Then, the assertion follows from the interpolation \cite[Theorem 1.3.4]{G2014}.
\end{proof}

\subsection{Noncommutative symmetric (quasi-)Banach spaces}\label{SubsNCOpSp}

Let $\mathcal{B}(\mathcal{H})$ denote the algebra of all bounded linear operators on a Hilbert space $\mathcal{H}$. Let us denote by $\mathcal{M}\subset \mathcal{B}(\mathcal{H})$ a semifinite von Neumann algebra equipped with a normal semifinite faithful trace $\tau$, and by $\mathcal{M}'$ we define commutant of $\mathcal{M}$:
\begin{equation*}
    \mathcal{M}'=\left\{y\in\mathcal{B}(\mathcal{H}): yx=xy \ \ \forall x\in\mathcal{M}  \right\}.
\end{equation*}
\begin{definition} A closed and densely defined operator $x$ on $\mathcal{H}$, with the domain ${\rm dom}(x)$, is called \textit{affiliated} with $\mathcal{M}$ if $u^{*}xu=x$ for any unitary $u$ from $\mathcal{M}'.$ The operator $x$ is said to be $\tau$-\textit{measurable} if $x$ is affiliated with $\mathcal{M}$ and for every $\varepsilon>0$, there exists a projection $p\in\mathcal{M}$ such that $p(\mathcal{H})\subset {\rm dom}(x)$ and $\tau(\textbf{1}-p)<\varepsilon$, here $\textbf{1}$ is identity operator on $\mathcal{H}$. 
\end{definition}
The set of all $\tau$-measurable operators will be denoted by $S(\mathcal{M})$.

For $x\in S(\mathcal{M})$, define the \textit{generalized $t$-th singular number} of $x$ by
\begin{equation}\label{GenSingNumb}
    \mu(t,x)=\inf\{s>0: \tau(e_{(s,\infty)}(|x|))\leq t\}, \ \  t\geq 0,
\end{equation}
where $e_{(s,\infty)}(|x|)$ is the spectral projection of $|x|$ corresponding to the interval $(s,\infty)$. The function $t \mapsto \mu(t, x)$ is decreasing and right-continuous. For more details on generalized $t$-th singular number, we refer to \cite[Chapter 3]{DdPS} and \cite[Chapter 2]{LSZ}.

\begin{definition}\label{NC Sym}
A linear subset $\mathcal{E} \subset S(\mathcal{M})$ equipped with a complete (quasi-)norm $\|\cdot\|_{\mathcal{E}}$ is called a \emph{(quasi-)Banach symmetric operator space} if, for every $x \in \mathcal{E}$ and every $y \in S(\mathcal{M})$ satisfying $\mu(y) \leq \mu(x)$, it follows that $y \in \mathcal{E}$ and $\|y\|_{\mathcal{E}} \leq \|x\|_{\mathcal{E}}$.
\end{definition}

In the case when the von Neumann algebra is commutative, specifically $\mathcal{M} = L_{\infty}(0,1)$, the construction of a symmetric operator space coincides with a symmetric function space on $(0,1)$.

The construction of noncommutative symmetric operator spaces proceeds as follows. Given a symmetric (quasi-)Banach function space $E$ on $(0,1)$, we define
$$
\mathcal{E}(\mathcal{M}) = \{ x \in S(\mathcal{M}) : \mu(x) \in E \}
$$
and endow it with the natural (quasi-)norm
\begin{equation*}\label{SymNCspaceNorm}
\|x\|_{\mathcal{E}(\mathcal{M})} := \|\mu(x)\|_E, \ \ x \in \mathcal{E}(\mathcal{M}).
\end{equation*}
Then $(\mathcal{E}(\mathcal{M}), \|\cdot\|_{\mathcal{E}(\mathcal{M})})$ forms a (quasi-)Banach space, called the \emph{noncommutative symmetric (quasi-)Banach operator space} associated with $(\mathcal{M},\tau)$ corresponding to $(E,\|\cdot\|_{E})$ \cite{KS, F}.

As established in \cite{KS}, \cite{F} (see also \cite[Chapter 3]{LSZ}), the mapping
$$
(E, \|\cdot\|_{E}) \longleftrightarrow (\mathcal{E}(\mathcal{M}), \|\cdot\|_{\mathcal{E}(\mathcal{M})})
$$
defines a one-to-one correspondence between symmetric operator spaces in $S(\mathcal{M})$ and symmetric function spaces in $S(\Omega)$.

\begin{definition}\label{NC FSym}
A (quasi-)Banach operator space $\mathcal{E}(\mathcal{M}) \subset S(\mathcal{M})$ is called \emph{fully symmetric} if for every $x \in \mathcal{E}(\mathcal{M})$ and every $y \in S(\mathcal{M})$ with $\mu(y) \prec\prec \mu(x)$, we have $y \in \mathcal{E}(\mathcal{M})$ and $\|y\|_{\mathcal{E}(\mathcal{M})} \leq \|x\|_{\mathcal{E}(\mathcal{M})}$.
\end{definition}

A more details of the properties of these spaces can be found in \cite{KS, LSZ}, and \cite{DdPS}.

In particular, if $E=L_{p}(0,1),$ $0< p<\infty,$ then we obtain
\begin{equation}\label{LponVNeumann}
    L_{p}(\mathcal{M})=\Big\{x\in S(\mathcal{M}):\ \mu(x)\in L_{p}(0,1)\Big\}
\end{equation}
with the quasi-norm
\begin{equation}\label{LponVNeumannNorm}
    \|x\|_{L_{p}(\mathcal{M})}:=\|\mu(x)\|_{L_{p}(0,1)}, \,\ x\in L_p(\mathcal{M}).
\end{equation}
If $p=\infty,$ then we set $L_{\infty}(\mathcal{M}):=\mathcal{M}$ with the uniform norm $\|x\|_{L_{\infty}(\mathcal{M})}:=\|x\|,$ $x\in \mathcal{M}.$

As shown in \cite{PXu}, $(L_{p}(\mathcal{M}), \|\cdot\|_{L_{p}(\mathcal{M})})$ is a Banach space when $1 \leq p \leq \infty$ and a quasi-Banach space when $0 < p < 1$.

Moreover, when $p=2$ the space $L_{2}(\mathcal{M})$ becomes Hilbert space with the inner product
\begin{equation*}
   \langle x,y \rangle:=\tau(xy^{\ast}),\,\, x,y\in L_{2}(\mathcal{M}). 
\end{equation*}

Let $\psi$ be a given increasing concave function on $(0,1)$ with $\psi(0)=0$. Define noncommutative Lorentz and Marcinkiewicz space associated with the von Neumann algebra $(\mathcal{M},\tau)$ as follows
\begin{equation*}
    \Lambda_{\psi}(\mathcal{M})=\left\{ x\in S (\mathcal{M}) : \int\limits_{0}^{\tau(\textbf{1})}\mu(s,x)\,d\psi(s)<\infty \right\},
\end{equation*}
and 
\begin{equation*}
    \mathtt{M}_{\psi}(\mathcal{M})=\left\{ f\in S (\mathcal{M}) :\sup_{0<t<\tau(\textbf{1})}\frac{1}{\psi(t)} \int\limits_{0}^{t}\mu(s,x)\,ds<\infty \right\}.
\end{equation*}

 For further details on the theory of noncommutative symmetric spaces over general semifinite von Neumann algebras $\mathcal{M}$, we refer to \cite{DdPS, LSZ, PXu}.

\subsection{Noncommutative  torus}
Given  $d \geq 2,$ let  $\mathbb{T}^{d}$ be the usual $d$-dimensional torus, that is, 
$\mathbb{T}^{d} = \mathbb{R}^{d} / \mathbb{Z}^{d}.$ Let $\theta = (\theta_{jk})_{1 \leq j,k \leq d}$ be a real antisymmetric matrix, i.e.
$$
\theta_{jk} = -\,\theta_{kj}, \ \ 1 \leq j,k \leq d,
$$ with  $\theta_{1}, \ldots, \theta_{d}$ being the column vectors of $\theta$. For $j = 1, \ldots, d$, we define unitary operators by
 $$
(U_{j}f)(t) = e^{it_{j}}f(t + \pi \theta_{j}), \,\ f \in L_{2}(\mathbb{T}^{d}).
 $$
These operators satisfy the commutation relations
$$
U_{k}U_{j} = e^{2\pi i \theta_{kj}} U_{j}U_{k}, \,\ j, k = 1, \ldots, d.
$$
Let  $\mathcal{A}_{\theta}$ denote the $C^{*}$-subalgebra of $\mathcal{B}(L_{2}(\mathbb{T}^{d}))$ generated by these unitary operators. Since $U_{1}, \ldots, U_{d}$ are unitary, we also have
$$
U_{j}U_{k}^{*} = e^{2\pi i \theta_{kj}} U_{k}^{*}U_{j}, \,\  j, k = 1, \ldots, d.
$$

For $n = (n_{1}, \ldots, n_{d}) \in \mathbb{Z}^{d},$ we define 
$$ 
e^{\theta}_{n} = U_{1}^{n_{1}} U_{2}^{n_{2}} \cdots U_{d}^{n_{d}}.
$$
A polynomial in  $\mathcal{A}_{\theta}$ is a finite sum
$$
x = \sum_{n\in \mathbb{Z}^{d}} \alpha_{n} e^{\theta}_{n}, \ \ \alpha_{n} \in \mathbb{C}, \ \ n \in \mathbb{Z}^{d},
$$
meaning  $\alpha_n = 0$ for all but finitely many $ n \in \mathbb{Z}^{d}$. The space of all such polynomials $\mathcal{P}_{\theta}$ is dense in $\mathcal{A}_{\theta}.$ 

For any $x$ as above, we define $\tau_{\theta}(x) = \alpha_{\mathbf{0}},$ where $\mathbf{0} = (0, \ldots, 0) \in \mathbb{Z}^{d}.$ This $\tau_{\theta}$ extends to a normal, faithful, finite tracial state on $\mathcal{A}_{\theta}$. Furthermore, we have
$$
\tau_{\theta}(xy) = \tau_{\theta}(yx), \,\ x, y \in \mathcal{A}_{\theta}.
$$

If we define an inner product on $\mathcal{A}_{\theta}$ by
$$
\langle x, y\rangle = \tau_{\theta}(xy^*).
$$
Then  $(\mathcal{A}_{\theta}, \langle\cdot, \cdot\rangle)$ becomes a pre-Hilbert space.

The von Neumann algebra  $\mathbb{T}^{d}_{\theta}$  is defined as the weak $*$-closure of  $\mathcal{A}_{\theta}$ in the GNS representation of $\tau_{\theta}.$ This algebra, $\mathbb{T}^{d}_{\theta},$ is referred to as the $d$-dimensional noncommutative (or quantum) torus. The trace $\tau_{\theta}$ extends to a normal, faithful, tracial state on $\mathbb{T}^{d}_{\theta},$ also denoted by $\tau_{\theta}.$ Note that the von Neumann algebra $\mathbb{T}^{d}_{\theta}$ is hyperfinite. For more details about the noncommutative torus, refer to \cite[Chapter 3, p. 190]{LSZ} and \cite{CXY, HaLeePonge, McDSX}.

In this paper, we primarily work with the von Neumann algebra $\mathcal{M}=L_{\infty}(\mathbb{T}^{d}_{\theta})$ equipped with the trace $\tau_{\theta}$.

We now introduce the noncommutative $L_{p}$-spaces associated with the von Neumann algebra $\mathbb{T}^{d}_{\theta}$ and the trace $\tau_{\theta}$.

For $0<p<\infty$, the $L_{p}$-(quasi-)norm is defined by
\begin{equation*}
    \|x\|_{L_{p}(\mathbb{T}^{d}_{\theta})} = \left( \tau_{\theta}(|x|^{p}) \right)^{\frac{1}{p}}, \ \ x \in \mathbb{T}^{d}_{\theta},
\end{equation*}
where $|x| = (x^{*}x)^{\frac{1}{2}}$. The space $L_{p}(\mathbb{T}^{d}_{\theta})$ is defined as the completion of $\{x \in \mathbb{T}^{d}_{\theta} : \|x\|_{L_{p}(\mathbb{T}^{d}_{\theta})} < \infty\}$ with respect to this (quasi-)norm. For $p = \infty$, we set $L_{\infty}(\mathbb{T}^{d}_{\theta}) = \mathbb{T}^{d}_{\theta}$ equipped with the operator norm. We define $L_{p}(\mathbb{T}^{d}_{\theta}), \ 0<p<\infty$ spaces as in \eqref{LponVNeumann}
\begin{equation}\label{LponTor}
    L_{p}(\mathbb{T}^{d}_{\theta})=\Big\{x\in S(\mathbb{T}^{d}_{\theta}):\ \mu(x)\in L_{p}(0,1)\Big\}, \ \ 0<p\leq\infty,
\end{equation}
with the corresponding quasi-norm
\begin{equation}\label{LponTorNorm}
    \|x\|_{L_{p}(\mathbb{T}^{d}_{\theta})}:=\|\mu(x)\|_{L_{p}(0,1)}, \ \ x\in L_{p}(\mathbb{T}^{d}_{\theta}), \,\, 0<p<\infty.
\end{equation}
$$\|x\|_{L_{\infty}(\mathbb{T}^{d}_{\theta})}=\mu(0,x), \,\ x \in L_{\infty}(\mathbb{T}^{d}_{\theta}). $$

Let $1 \leq p, q \leq \infty$. The Lorentz space on the noncommutative torus $\mathbb{T}^{d}_{\theta}$ is defined for $1 \leq q \leq \infty$ by
\begin{equation}
L_{p,q}(\mathbb{T}^{d}_{\theta}) = \left\{ x \in S(\mathbb{T}^{d}_{\theta}) : \mu(x) \in L_{p,q}(0,1) \right\},    
\end{equation}
and is equipped with the quasi-norm
\begin{equation}\label{LorentzNorm}
    \|x\|_{L_{p,q}(\mathbb{T}^{d}_{\theta})} = \left( \int\limits_{0}^{1} \left( t^{\frac{1}{p}} \mu(t,x) \right)^{q} \, \frac{dt}{t} \right)^{\frac{1}{q}}, \ \ q<\infty.
\end{equation}
For the case $q = \infty$, the weak-$L_{p}$ space is defined as
$$
L_{p,\infty}(\mathbb{T}^{d}_{\theta}) = \left\{ x \in S(\mathbb{T}^{d}_{\theta}) : \mu(x) \in L_{p,\infty}(0,1) \right\},
$$
and is endowed with the quasi-norm
\begin{equation}\label{weakL1}
\|x\|_{L_{p,\infty}(\mathbb{T}^{d}_{\theta})} = \sup_{t > 0} \, t^{\frac{1}{p}} \mu(t,x), \ \ x \in L_{p,\infty}(\mathbb{T}^{d}_{\theta}).
\end{equation}

As in the commutative setting, $\|\cdot\|_{L_{p,q}(\mathbb{T}^{d}_{\theta})}$ defines a quasi-norm for $1 \leq p, q \leq \infty$. Moreover, $(L_{p,q}(\mathbb{T}^{d}_{\theta}), \|\cdot\|_{L_{p,q}(\mathbb{T}^{d}_{\theta})})$ forms a quasi-Banach space with respect to this quasi-norm. Note that, analogous to the classical case, when $0< p \leq \infty$, we have $L_{p,p}(\mathbb{T}^{d}_{\theta}) = L_{p}(\mathbb{T}^{d}_{\theta})$ with equivalent norms (\cite[Chapter 6, p. 436 ]{DdPS}). 

The family $\{e^{\theta}_{n}\}_{n \in \mathbb{Z}^{d}}$ forms an orthonormal basis in $L_{2}(\mathbb{T}_{\theta}^{d})$ (see \cite{CXY, HaLeePonge}), satisfying the following properties:
\begin{equation}\label{orthogonal-relations}
\begin{split}
&\tau_{\theta}((e^{\theta}_{n})^{*}e^{\theta}_{k}) = \delta_{n,k}, \\
&\|e^{\theta}_{n}\|_{L_{\infty}(\mathbb{T}^{d}_{\theta})} = \|(e^{\theta}_{n})^{*}\|_{L_{\infty}(\mathbb{T}^{d}_{\theta})} = 1,
\end{split}
\end{equation}
where $\delta_{n,k}$ denotes the Kronecker delta. Furthermore, the inclusion $L_{q}(\mathbb{T}^{d}_{\theta}) \subset L_{p}(\mathbb{T}^{d}_{\theta})$ holds for all $0 < p \leq q \leq \infty$.
 
For $x \in L_{1}(\mathbb{T}^{d}_{\theta})$, we define its formal Fourier series by
\begin{equation*}
    x \sim \sum_{n \in \mathbb{Z}^{d}} \widehat{x}(n) e^{\theta}_{n},
\end{equation*}
where the Fourier coefficients are given by
\begin{equation}\label{eq.2.3}
    \widehat{x}(n) = \tau_{\theta}(x(e^{\theta}_{n})^{*}), \ \ n \in \mathbb{Z}^{d}.
\end{equation}
Every element $x \in L_{2}(\mathbb{T}^{d}_{\theta})$ admits a unique expansion
\begin{equation}\label{F-representation}
    x = \sum_{n \in \mathbb{Z}^{d}} \widehat{x}(n) e^{\theta}_{n}, \ \   \widehat{x}(n)=\tau_{\theta}(x(e^{\theta}_{n})^{*}),   \ \ n\in\mathbb{Z}^{d}.
\end{equation}

We have the Plancherel identity (see \cite[Formula (2.3)]{McDSX})
\begin{equation}\label{Parceveal}
    \|x\|_{L_{2}(\mathbb{T}^{d}_{\theta})} = \|\widehat{x}\|_{\ell_{2}(\mathbb{Z}^{d})}.
\end{equation}

The space $ C^{\infty}(\mathbb{T}^{d}_{\theta}) $ is defined by 
\begin{equation*} C^{\infty}(\mathbb{T}^{d}_{\theta})=\left\{x\in\mathcal{A}_{\theta} \, : \, x=\sum_{n \in \mathbb{Z}^{d}} \widehat{x}(n) e^{\theta}_{n}\text{ such that }
\sup\limits_{n\in\mathbb{Z}^{d}}(1+|n|)^k |\widehat{x}(n)|<\infty,  \forall k \in \mathbb{N}\right\},
\end{equation*}
where $|n|=(n_{1}^{2}+\cdots+n_{d}^{2})^{\frac{1}{2}}$, or equivalently, for each $k\in\mathbb{N}$ there exists a constant $C_{k}>0$ such that $|\widehat{x}(n)|\leq C_{k}(1+|n|)^{-k}$ for all $n\in\mathbb{Z}^{d}$ (see, \cite[p.1238]{McDSX} and \cite[p. 19]{XXY}). This space is regarded as the analogue of smooth functions on the noncommutative torus, as it coincides with the space of $ C^\infty $ functions in the commutative case, that is when $\theta=0$ (see \cite[p. 19]{XXY}). The $C^{\infty}(\mathbb{T}^{d}_{\theta})$ contains the all of polynomials on $\mathbb{T}^{d}_{\theta}$, i.e. $\mathcal{P}_{\theta}\subset C^{\infty}(\mathbb{T}^{d}_{\theta})$ (\cite[p. 19]{XXY}).

The space $C^{\infty}(\mathbb{T}^{d}_{\theta})$ is endowed with a canonical Fréchet topology. Its topological dual, denoted $ \mathcal{D}'(\mathbb{T}^{d}_{\theta}) $, is defined as the space of distributions on $ \mathbb{T}^{d}_{\theta} $ (\cite{McDSX, XXY}). For $x\in C^{\infty}(\mathbb{T}^{d}_{\theta})$ and $F\in \mathcal{D}'(\mathbb{T}^{d}_{\theta})$, we will use the $(\cdot , \cdot )$ to denote the duality between these two spaces: $(F,x)=F(x)$. For instance, by this duality structure, we can extend the Fourier series for $F\in\mathcal{D}'(\mathbb{T}^{d}_{\theta})$ as follows (\cite[p. 21]{XXY}):
\begin{equation*}
    \widehat{F}(n)=(F,(e_{n}^{\theta})^{*}), \ \ n\in\mathbb{Z}^{d}.
\end{equation*}
Then any $F\in\mathcal{D}'(\mathbb{T}^{d}_{\theta})$ admits the following Fourier series
\begin{equation}\label{FSonD}
    F=\sum_{n\in\mathbb{Z}^{d}}\widehat{F}(n)e^{\theta}_{n},
\end{equation}
which converges to $F$ in any (reasonable) summation method (\cite[p. 21]{XXY}).

The space of polynomials $\mathcal{P}_{\theta}$ is naturally embedded into $\mathcal{D}'(\mathbb{T}^{d}_{\theta})$ by identifying each $p \in \mathcal{P}_{\theta}$ with the distribution $x \mapsto \tau_{\theta}(p x)$ for $x \in C^{\infty}(\mathbb{T}^{d}_{\theta})$.

The spaces $L_{p}(\mathbb{T}^{d}_{\theta})$, $1\leq p\leq \infty$, can be viewed naturally as subspaces of the distribution space $\mathcal{D}'(\mathbb{T}^{d}_{\theta})$. More precisely, every element $y\in L_{p}(\mathbb{T}^{d}_{\theta})$ induces a continuous linear functional on $C^{\infty}(\mathbb{T}^{d}_{\theta})$ via (see, for example, \cite[p.~20]{XXY}):
\begin{equation*}
x \mapsto \tau_{\theta}(yx), \ x\in C^{\infty}(\mathbb{T}^{d}_{\theta})    
\end{equation*} 

For a noncommutative symmetric Banach function space $\mathcal{E}(\mathbb{T}^{d}_{\theta})$, one has the inclusions (see \cite[Theorem~4.4.6 (i), p. 273]{DdPS})
\begin{equation*}
L_{\infty}(\mathbb{T}^{d}_{\theta}) \subseteq \mathcal{E}(\mathbb{T}^{d}_{\theta}) \subseteq L_{1}(\mathbb{T}^{d}_{\theta}).   
\end{equation*}
Combining these embeddings with the above identification of $L_{p}(\mathbb{T}^{d}_{\theta})$ as subspaces of $\mathcal{D}'(\mathbb{T}^{d}_{\theta})$, we conclude that $\mathcal{E}(\mathbb{T}^{d}_{\theta})$ also embeds into $\mathcal{D}'(\mathbb{T}^{d}_{\theta})$.

\subsection{Fourier multipliers and convolution on noncommutative torus}

Let $\varphi = \{\varphi_{n}\}_{n \in \mathbb{Z}^{d}}$ be a bounded sequence of complex numbers. The Fourier multiplier $T_{\varphi}$ with symbol $\varphi$ is defined by 
\begin{equation}\label{eq.4.1}
    T_{\varphi}e_{n}^{\theta} = \varphi_{n} e_{n}^{\theta}, \,\ n\in \mathbb{Z}^{d}.
\end{equation}

If $\phi$ is a function of polynomial growth, then the associated Fourier multiplier $T_{\phi}$ is a continuous map on both $C^{\infty}(\mathbb{T}^{d}_{\theta})$ and
$\mathcal{D}'(\mathbb{T}^{d}_{\theta})$ (\cite[p. 21]{XXY}). For more discussion on Fourier multipliers on noncommutative torus, we refer the reader to \cite{CXY, McDSX, RST, STT, XXY}

For each $ s \in \mathbb{T}^{d} $, define $ \alpha_{s} $ to be the $*$-isomorphism of $ \mathbb{T}^{d}_{\theta}$ determined by  
\begin{equation*}
    \alpha_{s}(e_{n}^{\theta}) = s^{n} e_{n}^{\theta} = s_{1}^{n_{1}} \cdots s_{d}^{n_{d}}\: U_{1}^{n_{1}} \cdots U_{d}^{n_{d}}, \ \ n \in \mathbb{Z}^d.
\end{equation*}
The map $ \alpha_{s} $ preserves the trace $ \tau_{\theta} $ and thus extends to an isometry on $ L_{p}(\mathbb{T}^{d}_{\theta}) $ for every $ 0 < p \leq \infty $ (\cite{CXY, McDSX, XXY}). Consequently,
\begin{equation}\label{TransfNorm}
    \|\alpha_{s}(x)\|_{L_p(\mathbb{T}^{d}_\theta)} = \|x\|_{L_p(\mathbb{T}^{d}_\theta)}, \ \ x \in L_{p}(\mathbb{T}^{d}_\theta), \ \ 0 < p \leq \infty.
\end{equation}
\begin{definition}
    Let $f\in L_{1}(\mathbb{T}^{d})$ and $x\in L_{1}(\mathbb{T}^{d}_{\theta})$, then we define \textit{convolution} as follows (\cite[p. 1240]{McDSX})
\begin{equation}\label{Convolution}
    f\ast x=\int\limits_{\mathbb{T}^{d}}f(s)\alpha_{s}(x)ds,
\end{equation}
in the sense of Bochner integration in $L_{1}(\mathbb{T}^{d},L_{1}(\mathbb{T}^{d}_{\theta}))$.
 \end{definition}
  In terms of Fourier multipliers we have (\cite[Section 2.2]{McDSX})
\begin{equation}\label{ConvMult}
    f\ast x=T_{\widehat{f}} (x),
\end{equation}
where $\widehat{f}$ is the Fourier transform of $f$ on $\mathbb{T}^{d}$ (see, \cite[Chapter 3]{G2014}).

\subsection{Differential calculus for noncommutative torus}
Many aspects of harmonic analysis on the classical torus $\mathbb{T}^{d}$ extend naturally to the noncommutative torus $\mathbb{T}^{d}_{\theta}$ (\cite{HaLeePonge, McDSX, RST, STZ, XXY}). We adopt the standard differential structure on the quantum torus
$\mathbb{T}^{d}_\theta$ as in \cite[Section~2]{XXY}. For each $1\leq j\leq d$ we define a partial differentiation operators $\partial_{j}$ by
\begin{equation*}
    \partial_j(e^{\theta}_{n}) =  i\,n_j\, e^{\theta}_{n},
    \ \ n=(n_1,\dots,n_d)\in\mathbb{Z}^d.
\end{equation*}
Each $\partial_{j}$ may be viewed as a densely defined closed (unbounded) operator on $L_{2}(\mathbb{T}^{d}_{\theta})$, whose adjoint is $-\partial_{j}$. The Laplacian on $\mathbb{T}^{d}_\theta$ is defined by $ \Delta_{\theta}=\partial_{1}^{2}+\cdots+\partial_{d}^{2}$. Then $\Delta_{\theta} = -(\partial_1^*\partial_1 + \cdots + \partial_d^*\partial_d)$, so $-\Delta_{\theta}$ is a self-adjoint and positive operator on $L_{2}(\mathbb{T}^{d}_{\theta})$ with spectrum $\{|n|^2 : n \in \mathbb{Z}^{d}\}$ (\cite[Lemma 2.1, p. 19]{XXY}), and we have
\begin{equation}\label{NCLaplacian}
   - \Delta_{\theta}(e_{n}^{\theta})=|n|^{2}e_{n}^{\theta}, \ \ n\in\mathbb{Z}^{d}.
\end{equation}

By duality relation between $C^{\infty}(\mathbb{T}^{d}_\theta)$ and $\mathcal{D}'(\mathbb{T}^{d}_\theta)$, each $\partial_j$ extends to $\mathcal{D}'(\mathbb{T}^{d}_\theta)$ by the following formula
\begin{equation*}
    ( \partial_{j}F,x)=-( F,\partial_{j}x ), \ \ x\in C^{\infty}(\mathbb{T}^{d}_\theta), \ \ F\in \mathcal{D}'(\mathbb{T}^{d}_\theta).
\end{equation*}
Thus, $\Delta_{\theta}$ extends to $\mathcal{D}'(\mathbb{T}^{d}_\theta)$ as well
\begin{equation*}
    (\Delta_{\theta}F,x)=( F,\Delta_{\theta}x ), \ \ x\in C^{\infty}(\mathbb{T}^{d}_\theta), \ \ F\in \mathcal{D}'(\mathbb{T}^{d}_\theta).
\end{equation*}
For $s \in \mathbb{R}$, let $J^{s} $ denote the Bessel potential of order $s$, given by $(1 - \Delta_{\theta})^{\frac{s}{2}}$. Equivalently, via the functional calculus of the self-adjoint operator $-\Delta_{\theta}$, the Bessel potential can be defined as follows

\begin{equation}\label{J^s}
   J^{s} e_{n}^{\theta} = (1+|n|^2)^{\frac{s}{2}}e_{n}^{\theta},\  n\in\mathbb{Z}^{d}.
\end{equation}
Let $x\in C^{\infty}(\mathbb{T}^{d}_{\theta})$. If $s\le 0$, then the sequence $(1+|n|^{2})^{s/2}$, $n\in\mathbb{Z}^{d}$, is bounded by $1$ and tends to zero as $|n|\to\infty$; hence, multiplication by it preserves rapid decay. If $s>0$, the sequence $(1+|n|^{2})^{s/2}$ has a polynomial growth of order $s$, but rapid decay is still preserved. Indeed, for any $k\in\mathbb{N}$, choose $k'=k+\lceil s\rceil$. Then
\begin{equation*}
    (1+|n|^2)^{\frac{s}{2}}|\widehat{x}(n)|\leq C (1+|n|)^{s} C_{k'}(1+|n|)^{-k'}\leq C'(1+|n|)^{-k}, \,\ n\in\mathbb Z^d.
\end{equation*}
Hence, for every $s\in \mathbb R^d$, the operator $J^{s}$ maps $C^{\infty}(\mathbb{T}^{d}_\theta)$ into itself. By the foregoing, $J^{s}$ is extended to $\mathcal{D}'(\mathbb{T}^{d}_\theta)$
\begin{equation*}
    ( J^{s}  F,x)=( F,J^{s}  x), \ \ x\in C^{\infty}(\mathbb{T}^{d}_\theta), \ \ F\in \mathcal{D}'(\mathbb{T}^{d}_\theta).
\end{equation*}
Moreover, for each $s \in \mathbb{R}$ the operator $J^{s}$ is bijective on $\mathcal{D}'(\mathbb{T}^{d}_\theta)$ and it acts as a Fourier multiplier on $\mathcal{D}'(\mathbb{T}^{d}_\theta)$ with the symbol $(1+|n|^{2})^{\frac{s}{2}}, \ n\in\mathbb{Z}^{d}$ (see \cite[p. 1240]{McDSX} and \cite[p. 21]{XXY}).
Further, an element of $\mathcal{D}'(\mathbb{T}^{d}_\theta)$ is also denoted by $x$.

\subsection{Noncommutative fractional symmetric Sobolev spaces}\label{NCSobolevSpace}
Let $\mathcal{E}(\mathbb{T}^{d}_{\theta})$ be a noncommutative symmetric space on $\mathbb{T}^{d}_{\theta}.$ We define the Sobolev space associated with $\mathcal{E}(\mathbb{T}^{d}_{\theta}).$

For $s \in \mathbb{R}$, let $J^{s}  $ denote the Bessel potential of order $s$, given by \eqref{J^s}. The fractional symmetric Sobolev space of order $s \in \mathbb{R}$ is defined as
\begin{equation}\label{FractionalSobolev}
H^{s}_{\mathcal{E}}(\mathbb{T}^{d}_{\theta}) = \left\{ x \in \mathcal{D}'(\mathbb{T}^{d}_{\theta}) : J^{s}   x \in \mathcal{E}(\mathbb{T}^{d}_{\theta})\right\},
\end{equation}
such that
\begin{equation}\label{SobolevNorm}
\|x\|_{H^{s}_{\mathcal{E}}(\mathbb{T}^{d}_{\theta})} = \|J^{s}   x\|_{\mathcal{E}(\mathbb{T}^{d}_{\theta})}<\infty.  
\end{equation}

Since $J^0$ is the identity operator, we have $H^0_{\mathcal{E}}(\mathbb{T}^{d}_{\theta}) = \mathcal{E}(\mathbb{T}^{d}_{\theta})$. 
\begin{prop} $H^{s}_{\mathcal{E}}(\mathbb{T}^{d}_{\theta})$ is a Banach space.
\end{prop}
\begin{proof} It is clear from the definition that $H^{s}_{\mathcal{E}}(\mathbb{T}^{d}_{\theta})$ is a linear space. We first check the axioms of norm. Positivity and homogeneity of \eqref{SobolevNorm} are obvious. 

Let us take any $x\in H^{s}_{\mathcal{E}}(\mathbb{T}^{d}_{\theta})$. Then
$\|x\|_{H^{s}_{\mathcal{E}}(\mathbb{T}^{d}_{\theta})} = 0$ if and only if $J^{s}x = 0$
in $\mathcal{E}(\mathbb{T}^{d}_\theta)$. Since $\mathcal{E}(\mathbb{T}^{d}_\theta)\subset
\mathcal{D}'(\mathbb{T}^{d}_{\theta})$, this implies $J^{s}x=0$ in
$\mathcal{D}'(\mathbb{T}^{d}_{\theta})$. We show that $J^{s}x=0$ implies $x=0$.
Since $x\in H^{s}_{\mathcal{E}}(\mathbb{T}^{d}_{\theta})\subset \mathcal{D}'(\mathbb{T}^{d}_{\theta})$, we can write $x$ in terms of its Fourier coefficients in the sense of distributions (see, for example \cite[p. 21]{XXY}):
\begin{equation*}
    x \stackrel{\mathcal{D}'(\mathbb{T}^{d}_{\theta})}{=} \sum_{n\in\mathbb{Z}^d} \widehat x(n)\,e_n^\theta.
\end{equation*}
Then by the definition of $J^{s}$ \eqref{J^s}, we have
\begin{equation*}
    J^{s} x = \sum_{n\in\mathbb{Z}^d}
        (1+|n|^2)^{\frac{s}{2}} \widehat x(n)\,e_n^\theta.
\end{equation*}
In the commutative case $\theta=0$, distributions on $\mathbb{T}^d$ are uniquely determined by their Fourier coefficients (see \cite[Chapter III, §3.2.2(ii), p.~144]{SchmeisserTriebel} or \cite[Chapter III, Proposition~3.2.4, p.~184]{G2014}). The same argument applies verbatim in the noncommutative setting $\mathbb{T}^d_\theta$. Therefore, since $J^{s}x=0$ in $\mathcal{D}'(\mathbb{T}^{d}_{\theta}),$ it follows that all Fourier coefficients of $J^{s}x$ vanish 
\begin{equation*}
    (1+|n|^2)^{\frac{s}{2}} \widehat x(n) = 0
    \ \ \text{for all}\ \ n\in\mathbb{Z}^d.
\end{equation*}
However, $(1+|n|^2)^{\frac{s}{2}} > 0$  for all $n\in\mathbb{Z}^d$, this implies $\widehat x(n)=0$
for all $n\in\mathbb{Z}^d$, and hence $x=0$ in $\mathcal{D}'(\mathbb{T}^{d}_{\theta})$. Thus,
\begin{equation*}
    \|x\|_{H^{s}_{\mathcal{E}}(\mathbb{T}^{d}_\theta)}=0 \textnormal{  if and only if  } x = 0.
\end{equation*}
Finally, since $\mathcal{E}(\mathbb{T}^{d}_\theta)$ is a linear normed space, by linearity of
$J^{s}  $, for any $x,y\in H^{s}_{\mathcal{E}}(\mathbb{T}^{d}_\theta)$ we have,
\begin{equation*}
\begin{split}
    \|x+y\|_{H^{s}_{\mathcal{E}}(\mathbb{T}^{d}_\theta)}
    &= \|J^{s}  (x+y)\|_{\mathcal{E}(\mathbb{T}^{d}_\theta)}
    = \|J^{s}   x + J^{s}   y\|_{\mathcal{E}(\mathbb{T}^{d}_\theta)}\\
    & \leq \|J^{s}   x\|_{\mathcal{E}(\mathbb{T}^{d}_\theta)} + \|J^{s}   y\|_{\mathcal{E}(\mathbb{T}^{d}_\theta)}
    = \|x\|_{H^{s}_{\mathcal{E}}(\mathbb{T}^{d}_\theta)} + \|y\|_{H^{s}_{\mathcal{E}}(\mathbb{T}^{d}_\theta)}.
\end{split}
\end{equation*}
Thus, $H^{s}_{\mathcal{E}}(\mathbb{T}^{d}_\theta)$ is a linear normed space.

Now, let us show its completeness. Indeed, let $\{x_{k}\}_{k \geq 1}$ be a Cauchy sequence in $H^{s}_{\mathcal{E}}(\mathbb{T}^{d}_{\theta})$. By definition of the norm, the sequence $\{J^{s}x_{k}\}_{k\geq 1}$ is Cauchy sequence in $\mathcal{E}(\mathbb{T}^{d}_{\theta}).$ Since $\mathcal{E}(\mathbb{T}^{d}_{\theta})$ is complete, it follows that there exists $y_{s}\in\mathcal{E}(\mathbb{T}^{d}_{\theta})$ such that $J^{s}x_{k}\to y_{s}$ in $\mathcal{E}(\mathbb{T}^{d}_{\theta})$ as $k\to\infty$.
Since $ \mathcal{E}(\mathbb{T}^{d}_{\theta})$ is continuously embedded into $\mathcal{D}'(\mathbb{T}^{d}_{\theta}),$ it follows that $J^{s}x_{k}\to y_{s}$ as $k\to\infty$ in $\mathcal{D}'(\mathbb{T}^{d}_{\theta})$. In particular, for $s=0$ we have $J^{0}x_{k}=x_{k}\to y_{0}$ in $ \mathcal{E}(\mathbb{T}^{d}_{\theta})$ as $k\to\infty.$

Let $x\in C^{\infty}(\mathbb{T}^{d}_{\theta})$ be arbitrary. Then, we have
\begin{equation*}
( J^{s}x_{k},x )=(x_{k} ,J^{s}x )\to (y_{0},J^{s}x )=(J^{s}y_{0} ,x ),
\ \ k\to\infty.
\end{equation*}
On the other hand, since $J^{s}x_{k}\to y_{s}$ in $\mathcal{E}(\mathbb{T}^{d}_{\theta})$ and hence, in $\mathcal{D}'(\mathbb{T}^{d}_{\theta})$, we also have
\begin{equation*}
(J^{s}x_{k},x ) \to (y_{s},x),
\ \ k\to\infty.
\end{equation*}

Since $(J^{s}x_{k},x )=(x_{k} ,J^{s}x), \  k\geq 1,$ we obtain that $( J^{s}y_{0},x)=(y_{s} ,x)$  for all $x\in C^{\infty}(\mathbb{T}^{d}_{\theta})$, i.e.
\begin{equation*}
J^{s}y_{0}=y_{s}\, \text{ in }\mathcal{D}'(\mathbb{T}^{d}_{\theta}).
\end{equation*}
Thus $J^{s}y_{0}\in\mathcal{E}(\mathbb{T}^{d}_{\theta})$, so $y_{0}\in H^{s}_{\mathcal{E}}(\mathbb{T}^{d}_{\theta})$. Finally,
\begin{equation*}
\|x_{k}-y_{0}\|_{H^{s}_{\mathcal{E}}}
=\|J^{s}(x_{k}-y_{0})\|_{\mathcal{E}}
=\|J^{s}x_{k}-J^{s}y_{0}\|_{\mathcal{E}}=\|J^{s}x_{k}-y_{s}\|_{\mathcal{E}}
\to 0, \ k\to\infty, 
\end{equation*}
showing that $x_{k}\to y_{0}$ as $k\to\infty$ in $H^{s}_{\mathcal{E}}(\mathbb{T}^{d}_{\theta})$. Hence, $H^{s}_{\mathcal{E}}(\mathbb{T}^{d}_{\theta})$ is a Banach space.

\end{proof}

In particular, when $\mathcal{E}(\mathbb{T}^{d}_{\theta}) = L_p(\mathbb{T}^{d}_{\theta})$, we define the Sobolev space (\cite{McDSX, XXY}) of order $s \in \mathbb{R}$ on $\mathbb{T}^{d}_{\theta}$ by
$$
H_{L_{p}}^{s}(\mathbb{T}^{d}_{\theta})=H^{s}_{p}(\mathbb{T}^{d}_{\theta}) := \left\{ x\in \mathcal{D}'(\mathbb{T}^{d}_{\theta})  : J^{s}   x \in L_{p}(\mathbb{T}^{d}_{\theta}) \right\}.
$$

\subsection{Mittag-Leffler function and Caputo fractional derivative}
Let $\alpha,\beta\in\mathbb{C}$ be such that the real part of $\alpha$ is $\mathfrak{R}(\alpha)>0.$ 

The two-parameter Mittag-Leffler function is defined by the following formula
\begin{equation*}
E_{\alpha, \beta}(z) = \sum_{k=0}^{+\infty} \frac{z^k}{\Gamma(\alpha k + \beta)}, \ \ z \in \mathbb{C},
\end{equation*}
where $\Gamma(\cdot)$ is the classical Gamma function. In particular, if $\alpha$ is a real number such that $0<\alpha< 2$ and $\beta\in\mathbb{C},$  then we frequently use the following estimate
\begin{equation}\label{ML-estimate}
|E_{\alpha,\beta}(z)| \leq \frac{C}{1+|z|}, \ \  z\in\mathbb{C}, \ \ \gamma\leq |\arg(z)|\leq \pi, 
\end{equation}
where $\gamma$ is a real number satisfying  $\frac{\pi\alpha}{2}\leq\gamma\leq\min\{\pi, \pi\alpha\}$ and $C$ is a positive constant  (see \cite[Chapter 1, Theorem 1.4]{Podlubny}). Note that, $E_{\alpha,1}(\cdot)$ is briefly denoted by $E_{\alpha}(\cdot),$ that is the usual Mittag-Leffler function.
The series is absolutely and locally uniformly convergent for the given parameters (see, \cite[Chapter 1]{Podlubny}).

Now, let $0<\alpha<1$ and let $f$ be a continuous function on $(0,\infty)$, then by $\partial_{t}^{\alpha}$ we define Caputo fractional derivative (\cite[Chapter 2]{Podlubny}) of $f$: 

\begin{equation*} 
 \partial_{t}^{\alpha}f(t)=\frac{d}{dt}\frac{1}{ \Gamma\left( 1-\alpha \right)}\int\limits_{0}^{t} (t-s)^{-\alpha} f\left( s \right){{d}s}, \ \ t>0. 
\end{equation*}
It is known that, if $f$ is differentiable, then, when $\alpha\to1$, the Caputo fractional derivative tends to the classical derivative of the function $f$. In general, if $n\in\mathbb{N}$, then $\lim\limits_{\alpha\to n}\partial_{t}^{\alpha}f(t)=f^{(n)}(t), \ t>0$.

Let, $f$ and $g$ are given functions. We use the notation $f(t)\sim g(t)$ as $t\to 0$, if $f$ and $g$ \textit{asymptotically equivalent}, i.e. $\lim\limits_{t\to0}\frac{f(t)}{g(t)}=1$. We also write $A\lesssim B$ if there is a constant $c > 0$ such that $A \leq c B$. We write $A\asymp B$ if both
$A \lesssim B$ and $A \gtrsim B$ hold, possibly with different constants.

\section{Noncommutative fractional symmetric Sobolev spaces}
Now, we give some characterizations of fractional symmetric Sobolev space on noncommutative torus. Before that, let us prove one lemma about technical role of $J^{s}  $.

\begin{lem}\label{lem:J^s-iso}
For every $s \in \mathbb{R}$, the operator $J^{s}$ defined by \eqref{J^s} is a surjective linear isometry
\begin{equation*}
J^{s} : H^{s}_{\mathcal{E}}(\mathbb{T}^{d}_{\theta})
\to \mathcal{E}(\mathbb{T}^{d}_{\theta}),
\end{equation*}
with inverse given by $J^{-s} $.
\end{lem}

\begin{proof} Let us fix $s\in\mathbb{R}$. 
By the definition of $J^{s}$ is invertable in $\mathcal{D}'(\mathbb{T}^{d}_{\theta})$.
Then, for any $x \in H^{s}_{\mathcal{E}}(\mathbb{T}^{d}_{\theta})$ we have
\begin{equation*}
\|J^{s}   x\|_{\mathcal{E}(\mathbb{T}^{d}_{\theta})}
= \|x\|_{H^{s}_{\mathcal{E}}(\mathbb{T}^{d}_{\theta})},
\end{equation*}
so $J^{s}  $ is an isometry from $H^{s}_{\mathcal{E}}(\mathbb{T}^{d}_{\theta})$ into $\mathcal{E}(\mathbb{T}^{d}_{\theta})$.

To see that $J^{s}  $ is onto $\mathcal{E}(\mathbb{T}^{d}_{\theta})$, let $y \in \mathcal{E}(\mathbb{T}^{d}_{\theta})$ be arbitrary. Since $y \in \mathcal{D}'(\mathbb{T}^{d}_{\theta})$ and $J^{s}  $ is invertible on $\mathcal{D}'(\mathbb{T}^{d}_{\theta})$, we can define $x = J^{-s}  y \in \mathcal{D}'(\mathbb{T}^{d}_{\theta})$. Then
\begin{equation*}
J^{s}   x = J^{s}   J^{-s}  y = y \in \mathcal{E}(\mathbb{T}^{d}_{\theta}),
\end{equation*}
so $x \in H^{s}_{\mathcal{E}}(\mathbb{T}^{d}_{\theta})$ and $J^{s}   x = y$. Hence, $J^{s}  $ is surjective.

Therefore, $J^{s}$ is a surjective linear isometry between $H^{s}_{\mathcal{E}}(\mathbb{T}^{d}_{\theta})$ and $\mathcal{E}(\mathbb{T}^{d}_{\theta})$.
\end{proof}

\begin{thm}\label{thm:HEs-properties}
Let $s\in\mathbb{R}$. Let $E(0,1)$ be a separable symmetric Banach function space on $(0,1)$ and $\mathcal{E}(\mathbb{T}^{d}_{\theta})$ be the associated noncommutative symmetric Banach space on $\mathbb{T}^{d}_{\theta}$. Then
\begin{itemize}
    \item[(i)] A linear space of polynomials $\mathcal{P}_{\theta}$ is dense in
    $\mathcal{E}(\mathbb{T}^{d}_{\theta})$.
    
    \item[(ii)] A linear space of polynomials $\mathcal{P}_{\theta}$ is dense in
    $H^{s}_{\mathcal{E}}(\mathbb{T}^{d}_{\theta})$. 
    Consequently,
    $C^{\infty}(\mathbb{T}^{d}_{\theta})$ is dense in
    $H^{s}_{\mathcal{E}}(\mathbb{T}^{d}_{\theta})$.
    
    \item[(iii)] $H^{s}_{\mathcal{E}}(\mathbb{T}^{d}_{\theta})$ is separable.
\end{itemize}
\end{thm}
\begin{proof} 
(i) Fix $x \in L_{\infty}(\mathbb{T}^{d}_{\theta})$, $x \neq 0$, and let $x_{k} = F_{k}[x]$ be the square $k$-th Fejér mean of $x,$ that is,
\begin{equation*}
x_{k} = F_{k}[x] = \sum_{\substack{n \in \mathbb{Z}^{d} \\ \max\limits_{1 \leq j \leq d} |n_{j}| \leq k}} \prod_{j=1}^{d} \Big(1 - \frac{|n_{j}|}{k+1}\Big)\,\widehat{x}(n)\,e^{\theta}_{n}, \ \ k \geq 0.
\end{equation*}
By construction, each $x_{k}$ is a finite polynomial, hence $x_{k} \in \mathcal{P}_{\theta}$ for all $k \geq 0$.

By the properties of the Fejér means on the noncommutative torus (see, e.g., \cite[Proposition~3.1]{CXY}), we have
\begin{equation*}
x_{k} \to x \quad \text{in } L_{1}(\mathbb{T}^{d}_{\theta}) \quad \text{as } k \to \infty,
\end{equation*}
and
\begin{equation*}
\|x_{k}\|_{L_{\infty}(\mathbb{T}^{d}_{\theta})}
\leq \|x\|_{L_{\infty}(\mathbb{T}^{d}_{\theta})}, \ \ k \geq 0.
\end{equation*}
Define
\begin{equation*}
y_{k}
= \frac{x_{k} - x}{2\|x\|_{L_{\infty}(\mathbb{T}^{d}_{\theta})}}, \ \ k \geq 0.
\end{equation*}
Then, it is obvious that
\begin{equation*}
\|y_{k}\|_{L_{\infty}(\mathbb{T}^{d}_{\theta})} \leq 1 \ \text{ for all } \ k \geq 0,
\ \
\|y_{k}\|_{L_{1}(\mathbb{T}^{d}_{\theta})} \to 0 \ \text{ as } \ k \to \infty.
\end{equation*}

Since $E(0,1)$ is separable, by \cite[Lemma~6]{STZ} there exists a function $\psi_{\mathcal{E}} : [0,\infty) \to [0,\infty)$ with $\psi_{\mathcal{E}}(+0) = 0$ such that for all $k\geq 0,$ we have
\begin{equation*}
\|y_k\|_{\mathcal{E}(\mathbb{T}^{d}_{\theta})}
\leq \psi_{\mathcal{E}}\big(\|y_k\|_{L_{1}(\mathbb{T}^{d}_{\theta})}\big), \,\ 
\end{equation*}
Consequently, we obtain
\begin{equation*}
0 \leq \|y_{k}\|_{\mathcal{E}(\mathbb{T}^{d}_{\theta})}
\leq \psi_{\mathcal{E}}\big(\|y_{k}\|_{L_{1}(\mathbb{T}^{d}_{\theta})}\big)
\to 0 \ \text{ as } \ k \to \infty.
\end{equation*}
Therefore,
\begin{equation*}
\|x_{k} - x\|_{\mathcal{E}(\mathbb{T}^{d}_{\theta})}
= 2\|x\|_{L_{\infty}(\mathbb{T}^{d}_{\theta})}\,\|y_{k}\|_{\mathcal{E}(\mathbb{T}^{d}_{\theta})}
\to 0, \ k \to \infty.
\end{equation*}
Thus, for every $x \in L_{\infty}(\mathbb{T}^{d}_{\theta})$ there exists a sequence $\{x_{k}\}_{k \geq 0} \subset \mathcal{P}_{\theta}$ with $x_{k} \to x$ in $\mathcal{E}(\mathbb{T}^{d}_{\theta})$.

On the other hand, since $E(0,1)$ is separable, $L_{\infty}(\mathbb{T}^{d}_{\theta})$ is dense in $\mathcal{E}(\mathbb{T}^{d}_{\theta})$ by \cite[Lemma~7]{STZ}. Hence, $\mathcal{P}_{\theta}$ is dense in $\mathcal{E}(\mathbb{T}^{d}_{\theta})$, which proves (ii).

(ii) We first show that $\mathcal{P}_{\theta}$ is dense in $H^{s}_{\mathcal{E}}(\mathbb{T}^{d}_{\theta})$.
By Lemma~\ref{lem:J^s-iso}, the operator
\begin{equation*}
J^{s}   : H^{s}_{\mathcal{E}}(\mathbb{T}^{d}_{\theta})
\to \mathcal{E}(\mathbb{T}^{d}_{\theta})
\end{equation*}
is a surjective linear isometry with inverse $J^{-s}$. Let $x\in\mathcal{P}_{\theta}$ be a polynomial. Then, by the definition of $J^{s}$ in \eqref{J^s}, we obtain
\begin{equation*}
    J^{s}x=\sum_{n\in\mathbb{Z}^{d}}(1+|n|^{2})^{\frac{s}{2}}\alpha_{n}e_{n}^{\theta}.
\end{equation*}
Since the resulting series still consists of finitely many terms, $J^{s}x$ is again a polynomial, and therefore, $J^{s}(\mathcal{P}_{\theta}) \subset \mathcal{P}_{\theta}.$ 

Let $x \in H^{s}_{\mathcal{E}}(\mathbb{T}^{d}_{\theta})$ be arbitrary and set
\begin{equation*}
y = J^{s}   x \in \mathcal{E}(\mathbb{T}^{d}_{\theta}).
\end{equation*}
By (i), there exists a sequence $\{y_{n}\}_{n \geq 1} \subset \mathcal{P}_{\theta}$ such that
\begin{equation*}
\|y_{n} - y\|_{\mathcal{E}(\mathbb{T}^{d}_{\theta})} \to 0 \ \text{ as } \ n \to \infty.
\end{equation*}
Define
\begin{equation*}
x_{n} = J^{-s}  y_{n}, \ \ n \geq 1.
\end{equation*}
Since each $y_{n}$ is a polynomial, we have $x_{n} \in \mathcal{P}_{\theta}$ for all $n\geq 1.$ Furthermore,
\begin{equation*}
\|x_{n} - x\|_{H^{s}_{\mathcal{E}}(\mathbb{T}^{d}_{\theta})}
= \|J^{s}  (x_{n} - x)\|_{\mathcal{E}(\mathbb{T}^{d}_{\theta})}
= \|y_{n} - y\|_{\mathcal{E}(\mathbb{T}^{d}_{\theta})}
\to 0.
\end{equation*}
Thus $\mathcal{P}_{\theta}$ is dense in $H^{s}_{\mathcal{E}}(\mathbb{T}^{d}_{\theta})$.
Since, $ \mathcal{P}_{\theta} \subset C^{\infty}(\mathbb{T}^{d}_{\theta})$ and $\mathcal{P}_{\theta}$ is already dense in $H^{s}_{\mathcal{E}}(\mathbb{T}^{d}_{\theta})$, it follows that $C^{\infty}(\mathbb{T}^{d}_{\theta})$ is dense in $H^{s}_{\mathcal{E}}(\mathbb{T}^{d}_{\theta})$.

(iii) Since $E(0,1)$ is separable, $\mathcal{E}(\mathbb{T}^{d}_{\theta})$ is separable (\cite[Lemma~7]{STZ}) as a Banach space (for more details on separability properties of noncommutative symmetric spaces we refer to \cite[Section 1]{CzerKamin}). By Lemma~\ref{lem:J^s-iso}, the operator
\begin{equation*}
J^{s}   : H^{s}_{\mathcal{E}}(\mathbb{T}^{d}_{\theta})
\to \mathcal{E}(\mathbb{T}^{d}_{\theta})
\end{equation*}
is a surjective linear isometry of Banach spaces. It follows that $H^{s}_{\mathcal{E}}(\mathbb{T}^{d}_{\theta})$ is separable.

Alternatively, by (ii) the polynomial algebra $\mathcal{P}_{\theta}$ is dense in $H^{s}_{\mathcal{E}}(\mathbb{T}^{d}_{\theta})$. Since $\mathcal{P}_{\theta}$ is generated by the countable family $\{e^{\theta}_{n} : n \in \mathbb{Z}^{d}\}$, it contains a countable dense subset.
This also shows that $H^{s}_{\mathcal{E}}(\mathbb{T}^{d}_{\theta})$ is separable.
\end{proof}

\section{Noncommutative fractional Sobolev inequality} 

In this section we investigate the fractional Sobolev type inequality in noncommutative  symmetric spaces. First, let us prove some technical lemmas, which will be used to obtain the main result.

Recall that the convolution \eqref{Convolution} of $f\in L_{1}(\mathbb{T}^{d})$ and $x\in L_{1}(\mathbb{T}^{d}_{\theta})$ is defined by Bochner integral
\begin{equation*}
    f\ast x=\int\limits_{\mathbb{T}^{d}} f(s)\,\alpha_{s}(x)\,ds.
\end{equation*}
\begin{lem}\label{FirstTechLemma} Let $f\in L_{1}(\mathbb{T}^{d})$. If $x\in L_{1}(\mathbb{T}^{d}_{\theta}),$ then we have
\begin{equation}\label{ConvIneqFirst}
    \|f\ast x\|_{L_{1}(\mathbb{T}^{d}_{\theta})}\leq\|f\|_{L_{1}(\mathbb{T}^{d})}\|x\|_{L_{1}(\mathbb{T}^{d}_{\theta})},
\end{equation}
and if $x\in L_{\infty}(\mathbb{T}^{d}_{\theta}),$ then we have 
\begin{equation}\label{ConvIneqSecond}
    \|f\ast x\|_{L_{\infty}(\mathbb{T}^{d}_{\theta})}\leq \|f\|_{L_{1}(\mathbb{T}^{d})}\|x\|_{L_{\infty}(\mathbb{T}^{d}_{\theta})}.
\end{equation}
\end{lem}
\begin{proof} Let us fix $f\in L_{1}(\mathbb{T}^{d})$. By the property \eqref{TransfNorm} of $\alpha_{s},$ if $x\in L_{1}(\mathbb{T}^{d}_{\theta})$, we obtain
    \begin{equation*}\begin{split}
        \|f\ast x\|_{L_{1}(\mathbb{T}^{d}_{\theta})}&=\left\|\int\limits_{\mathbb{T}^{d}}f(s)\alpha_{s}(x)ds\right\|_{L_{1}(\mathbb{T}^{d}_{\theta})}\\
        & \leq  \int\limits_{\mathbb{T}^{d}}|f(s)|\tau_{\theta}(|\alpha_{s}(x)|)ds=\int\limits_{\mathbb{T}^{d}}|f(s)|\|\alpha_{s}(x)\|_{L_{1}(\mathbb{T}^{d}_{\theta})}ds\\
        &\stackrel{\eqref{TransfNorm}}{=}\int\limits_{\mathbb{T}^{d}}|f(s)|\|x\|_{L_{1}(\mathbb{T}^{d}_{\theta})}ds=\|f\|_{L_{1}(\mathbb{T}^{d})}\|x\|_{L_{1}(\mathbb{T}^{d}_{\theta})},
    \end{split}
    \end{equation*}
and similarly, if $x\in L_{\infty}(\mathbb{T}^{d}_{\theta}),$ then we obtain
 \begin{equation*}\begin{split}
        \|f\ast x\|_{L_{\infty}(\mathbb{T}^{d}_{\theta})}&=\left\|\int\limits_{\mathbb{T}^{d}}f(s)\alpha_{s}(x)ds\right\|_{L_{\infty}(\mathbb{T}^{d}_{\theta})}\\
        & \leq  \int\limits_{\mathbb{T}^{d}}|f(s)|\|\alpha_{s}(x)\|_{L_{\infty}(\mathbb{T}^{d}_{\theta})}ds\stackrel{\eqref{TransfNorm}}{=}\int\limits_{\mathbb{T}^{d}}|f(s)|\|x\|_{L_{\infty}(\mathbb{T}^{d}_{\theta})}ds\\
        &=\|f\|_{L_{1}(\mathbb{T}^{d})}\|x\|_{L_{\infty}(\mathbb{T}^{d}_{\theta})}.
    \end{split}
    \end{equation*}
    This completes the proof.
\end{proof}

\begin{lem}\label{first oneal lemma} Let $f\in L_{1}(\mathbb{T}^{d})$ and $x\in L_{1}(\mathbb{T}^{d}_{\theta})$. Then we have
$$f\ast x\prec\prec \|f\|_{L_{1}(\mathbb{T}^{d})}x.$$
\end{lem}
\begin{proof} Fix $f\in L_{1}(\mathbb{T}^{d})$ and without loss of generality, assume that $\|f\|_{L_{1}(\mathbb{T}^{d})}=1.$ It follows from Lemma \ref{FirstTechLemma} that, the mapping $x\mapsto f\ast x$ is a contraction on both $L_1(\mathbb{T}^{d}_{\theta})$ and $L_{\infty}(\mathbb{T}^{d}_{\theta}).$ Hence, the assertion follows from the interpolation \cite[Theorem 7.12.6]{DdPS}.
\end{proof}
\begin{lem}\label{questionable oneal lemma} If $f\in L_{\infty}(\mathbb{T}^{d})$ and if $x\in L_{1}(\mathbb{T}^{d}_{\theta}),$ then 
\begin{equation}\label{Eq:questionable oneal lemma}
   \|f\ast x\|_{L_{\infty}(\mathbb{T}^{d}_{\theta})}\leq 16 \|f\|_{L_{\infty}(\mathbb{T}^{d})}\|x\|_{L_{1}(\mathbb{T}^{d}_{\theta})}. 
\end{equation}
\end{lem}
\begin{proof} Let $f\in L_{\infty}(\mathbb{T}^{d})$, it is known that we can write $f$ as $f=f_1-f_2+if_3-if_4,$ where $f_k\geq0, k=1,2,3,4,$ and holds
\begin{equation}\label{funcDecom}
  \|f_k\|_{L_{\infty}(\mathbb{T}^{d})}\leq  \|f\|_{L_{\infty}(\mathbb{T}^{d})}, \,\ f\in L_{\infty}(\mathbb{T}^{d}). 
\end{equation}
Now, note that, every operator $x\in S(\mathbb{T}^{d}_{\theta})$ can be decomposed \cite[Chapter II, pp. 60-61]{DdPS} in the following way
\begin{equation*}
    x=\mathfrak{R}(x)+i\mathfrak{I}(x),
\end{equation*}
where $\mathfrak{R}(x)$ and $\mathfrak{I}(x)$ are self-adjoint operators and are determined by the following formulas
\begin{equation*}
    \mathfrak{R}(x)=\frac{x+x^{*}}{2},  \ \ \mathfrak{I}(x)=\frac{x-x^{*}}{2i}.
\end{equation*}
Moreover, every self-adjoint operator $x=x^{*}$ also has a decomposition as a sum of positive operators
\begin{equation*}
    x=x_{+}+x_{-},
\end{equation*}
where $x_{+}=\frac{x+|x|}{2},$ and $x_{-}=\frac{x-|x|}{2},$ $(|x|^{2}:=x^{*}x)$. Therefore, any $x\in S(\mathbb{T}^{d}_{\theta})$ can be expressed as a linear combination of four positive operators (see, \cite[Chapter I, p.33]{LSZ})
$$
    x=x_{1}-x_{2}+ix_{3}-ix_{4}, \,\ x_{l}\in S(\mathbb{T}^{d}_{\theta}).
$$
Let $x\in L_{1}(\mathbb{T}^{d}_{\theta})$ and $x=x_{1}-x_{2}+ix_{3}-ix_{4}.$ Then these positive components satisfy 
\begin{equation}\label{operDecom}\|x_{l}\|_{L_{1}(\mathbb{T}^{d}_{\theta})}\leq\|x\|_{L_{1}(\mathbb{T}^{d}_{\theta})},\,\ l=1,2,3,4.
\end{equation}
 By the triangle inequality, we have
\begin{equation*}
  \|f\ast x\|_{L_{\infty}(\mathbb{T}^{d}_{\theta})}\leq \sum_{k,l=1}^4\|f_k\ast x_l\|_{L_{\infty}(\mathbb{T}^{d}_{\theta})}.  
\end{equation*}
By the definition of convolution, if $f\geq0$ and $x\geq0,$ then $f\ast x\geq0,$
consequently, we have 
\begin{equation*}
  0\leq f_k\ast x_l\leq \|f_k\|_{L_{\infty}(\mathbb{T}^{d})}\cdot (1\ast x_l), \ \  k,l=1,2,3,4,  
\end{equation*}
where $1\ast x_{l}=\int\limits_{\mathbb{T}^{d}}\alpha_{s}(x_{l})ds=\tau_{\theta}(x_{l})$ (see, \cite[p. 1237]{McDSX}).  Hence,
\begin{equation*}
  \|f_k\ast x_l\|_{L_{\infty}(\mathbb{T}^{d}_{\theta})}\leq \|f_k\|_{L_{\infty}(\mathbb{T}^{d})}\|x_l\|_{L_{1}(\mathbb{T}^{d}_{\theta})}\leq\|f\|_{L_{\infty}(\mathbb{T}^{d})}\|x\|_{L_{1}(\mathbb{T}^{d}_{\theta})},\ \  k,l=1,2,3,4.  
\end{equation*}
Applying formulas \eqref{funcDecom} and \eqref{operDecom}, we get  
\begin{equation*}
   \|f\ast x\|_{L_{\infty}(\mathbb{T}^{d}_{\theta})}\leq 16 \|f\|_{L_{\infty}(\mathbb{T}^{d})}\|x\|_{L_{1}(\mathbb{T}^{d}_{\theta})}, \,\ x\in L_{1}(\mathbb{T}_{\theta}^{d}). 
\end{equation*}
This concludes the proof.
\end{proof}

\begin{cond}\label{main condition general} Let $\mathcal{M}_{1}, \mathcal{M}_{2}$ and $\mathcal{M}_{3}$ are finite von Neumann algebras equipped, respectively, with normal finite faithful traces $\tau_{1}, \tau_{2}$ and $\tau_{3}$. Consider bilinear mapping $\mathfrak{B}:L_{1}(\mathcal{M}_{1})\times L_{1}(\mathcal{M}_{2})\to L_{1}(\mathcal{M}_{3})$ such that
\begin{equation}\label{Gassump-1}
\|\mathfrak{B}(x,y)\|_{L_{1}(\mathcal{M}_{3})}\leq\|x\|_{L_{1}(\mathcal{M}_{1})}\|y\|_{L_{1}(\mathcal{M}_{2})},
\end{equation}
\begin{equation}\label{Gassump-2}
\|\mathfrak{B}(x,y)\|_{L_{\infty}(\mathcal{M}_{3})}\leq\|x\|_{L_{1}(\mathcal{M}_{1})}\|y\|_{L_{\infty}(\mathcal{M}_{2})},
\end{equation}
\begin{equation}\label{Gassump-3}
\|\mathfrak{B}(x,y)\|_{L_{\infty}(\mathcal{M}_{3})}\leq\|x\|_{L_{\infty}(\mathcal{M}_{1})}\|y\|_{L_{1}(\mathcal{M}_{2})}.
\end{equation}
\end{cond}

The following result is a noncommutative version of the famous O'Neil inequality in \cite[Lemma 1.4]{ONeil}.
\begin{lem} \label{final B-lemma} Suppose $\mathfrak{B}$ satisfies Condition \ref{main condition general}. For every $x\in L_1(\mathcal{M}_1,\tau_1)$ and $y\in \mathcal{M}_2$ (or $x\in\mathcal{M}_2$ and $y\in L_1(\mathcal{M}_2,\tau_2)$), we have
    \begin{equation*}
          \int\limits_0^t\mu(s,\mathfrak{B}(x,y))ds\leq \int\limits_0^t\mu(s,x)ds\cdot\int\limits_0^t\mu(s,y)ds+t\int\limits_t^1\mu(s,x)\mu(s,y)ds,\quad t\in(0,1). 
    \end{equation*}
\end{lem}
\begin{proof} Fix $x\in L_1(\mathcal{M}_1,\tau_1)$ and $y\in \mathcal{M}_2$ (or $x\in\mathcal{M}_2$ and $y\in L_1(\mathcal{M}_2,\tau_2)$). Fix some trace preserving $\ast$-homomorphisms $i_1:L_{\infty}(0,1)\to \mathcal{M}_1$ and $i_2:L_{\infty}(0,1)\to \mathcal{M}_2 $ such that $i_1(\mu(x))=|x|$ and $i_2(\mu(y))=|y|.$ Now, let $x=u_1|x|,$ $y=u_2|y|$ and $\mathfrak{B}(x,y)=u\cdot |\mathfrak{B}(x,y)|$ be the polar decompositions of $x,$ $y$ and $\mathfrak{B}(x,y),$ respectively. Let $\mathcal{A}_{3}\subset \mathcal{M}_3$ be an atomless abelian von Neumann subalgebra such that $|\mathfrak{B}(x,y)|$ is affiliated to $\mathcal{A}_{3}.$ Since $\mathcal{A}_{3}$ is an atomless abelian von Neumann subalgebra of $\mathcal{M}_{3},$ it follows that $\mathcal{A}_{3}$ is isomorphic to $L_{\infty}(0,1)$ (see, \cite[Chapter 4, Lemma 4.4.12, p. 282]{DdPS}). Fix a trace preserving $\ast$-isomorphism $i_3:\mathcal{A}_3\to L_{\infty}(0,1).$

Define the map $\mathfrak{B}_0:L_{\infty}(0,1)\times L_{\infty}(0,1)\to L_{\infty}(0,1)$ by setting
\begin{equation*}
\mathfrak{B}_0(f,g)=i_{3}\Big(\mathbb{E}_{\mathcal{A}_{3}}(u^{\ast}\cdot \mathfrak{B}(u_1\cdot i_1(f), u_2\cdot i_2(g)))\Big), \ \ f,g\in L_{\infty}(0,1),
\end{equation*}
where $\mathbb{E}_{\mathcal{A}_{3}}:\mathcal{M}_{3}\to \mathcal{A}_{3}$ is a conditional expectation \cite[Chapter 7, p. 456]{DdPS}. Since $\mathfrak{B}$ is bilinear and $\mathbb{E}_{\mathcal{A}_{3}}$ is linear, it follows that $\mathfrak{B}_{0}$ is bilinear.  We have
\begin{equation*}
\begin{split}
\|\mathfrak{B}_0(f,g)\|_{L_{\infty}(0,1)}&=\|\mathbb{E}_{\mathcal{A}_{3}}(u^{\ast}\cdot \mathfrak{B}(u_1\cdot i_1(f), u_2\cdot i_2(g))) \|_{L_{\infty}(\mathcal{A}_3,\tau_3)}\\
&\leq \|u^{\ast}\cdot \mathfrak{B}(u_1\cdot i_1(f), u_2\cdot i_2(g))\|_{L_{\infty}(\mathcal{M}_{3},\tau_3)} \\
   &\leq \|\mathfrak{B}(u_1\cdot i_1(f), u_2\cdot i_2(g))\|_{L_{\infty}(\mathcal{M}_{3},\tau_3)}. 
\end{split}
\end{equation*}

Hence, by \eqref{Gassump-2}, we obtain 
\begin{equation*}
\begin{split}
\| \mathfrak{B}_0(f,g)\|_{L_{\infty}(0,1)}
   &\leq \|\mathfrak{B}(u_1\cdot i_1(f), u_2\cdot i_2(g))\|_{L_{\infty}(\mathcal{M}_{3},\tau_3)}\\
   &\stackrel{\eqref{Gassump-2}}{\leq} \|u_{1}\cdot i_{1}(f)\|_{L_{1}(\mathcal{M}_{1},\tau_1)} \|u_{2}\cdot i_{2}(g)\|_{L_{\infty}(\mathcal{M}_{2},\tau_2)}\\
   &\leq\|i_{1}(f)\|_{L_{1}(\mathcal{M}_{1},\tau_1)} \| i_{2}(g)\|_{L_{\infty}(\mathcal{M}_{2},\tau_2)}\\
   &=\|f\|_{L_{1}(0,1)} \| g\|_{L_{\infty}(0,1)}.
\end{split}
\end{equation*}
Similarly, by \eqref{Gassump-3}, we have
\begin{equation*}
\begin{split}
   \| \mathfrak{B}_0(f,g)\|_{L_{\infty}(0,1)}
   \leq \|f\|_{L_{\infty}(0,1)} \| g\|_{L_{1}(0,1)}.
\end{split}
\end{equation*}
And, finally, since $\mathbb{E}_{\mathcal{A}_{3}}$ can be extended to a contraction from $L_{1}(\mathcal{M}_{3},\tau_3)$ to $L_{1}(\mathcal{A}_{3},\tau_3)$ (\cite[Chapter 7, Proposition 7.1.1, p. 457]{DdPS}), it follows from \eqref{Gassump-1} that
\begin{equation*}
\begin{split}
\|\mathfrak{B}_0(f,g)\|_{L_{1}(0,1)}&=\| \mathbb{E}_{\mathcal{A}_{3}}(u^{\ast}\cdot \mathfrak{B}(u_1\cdot i_1(f), u_2\cdot i_2(g))) \|_{L_{1}(\mathcal{A}_3,\tau_3)}\\
&\leq\|u^{\ast}\cdot \mathfrak{B}(u_1\cdot i_1(f), u_2\cdot i_2(g))\|_{L_{1}(\mathcal{M}_{3},\tau_3)} \\
&\leq\|\mathfrak{B}(u_1\cdot i_1(f), u_2\cdot i_2(g))\|_{L_{1}(\mathcal{M}_{3},\tau_3)}\\
& \stackrel{\eqref{Gassump-1}}{\leq} \|u_{1}\cdot i_{1}(f)\|_{L_{1}(\mathcal{M}_{1},\tau_1)}\| u_{2}\cdot i_{2}(g)\|_{L_{1}(\mathcal{M}_{2},\tau_2)}\\
&=\|f\|_{L_{1}(0,1)}\|g\|_{L_{1}(0,1)}.
\end{split}
\end{equation*}
Therefore, $\mathfrak{B}_0$ fits into the classical O'Neil setting \cite[Lemma 1.5]{ONeil}, and for any $f,g\in L_{\infty}(0,1)$, we have
\begin{equation}\label{oneil inequality}
  \int\limits_0^t\mu(s,\mathfrak{B}_0(f,g))ds\leq \int\limits_0^t\mu(s,f)ds\cdot\int\limits_0^t\mu(s,g)ds+t\int\limits_t^1\mu(s,f)\mu(s,g)ds,\quad t\in(0,1).  
\end{equation}
Recall that
$$u_1\cdot i_1(\mu(x))=x,\ \ u_2\cdot i_2(\mu(y))=y,\ \ u^{\ast}\cdot\mathfrak{B}(x,y)=|\mathfrak{B}(x,y)|.$$
Thus,
\begin{equation*}
  \begin{split}
\mathfrak{B}_0(\mu(x),\mu(y))&=i_{3}\Big(\mathbb{E}_{\mathcal{A}_{3}}(u^{\ast}\cdot \mathfrak{B}(u_1\cdot i_1(\mu(x)), u_2\cdot i_2(\mu(y))))\Big)\\
&=i_{3}\Big(\mathbb{E}_{\mathcal{A}_{3}}(u^{\ast}\cdot \mathfrak{B}(x,y))\Big)=i_{3}\Big(\mathbb{E}_{\mathcal{A}_{3}}(|\mathfrak{B}(x,y)|)\Big).
  \end{split}  
\end{equation*}
Since $|\mathfrak{B}(x,y)|$ is affiliated to $\mathcal{A}_3,$ it follows that
\begin{equation*}
  \mathfrak{B}_0(\mu(x),\mu(y))=i_{3}\Big(|\mathfrak{B}(x,y)|\Big).  
\end{equation*}
In particular,
\begin{equation*}
\mu\big(\mathfrak{B}_0(\mu(x),\mu(y))\big)=\mu\big(\mathfrak{B}(x,y)\big).  
\end{equation*}
Using \eqref{oneil inequality} with $f=\mu(x)$ and $g=\mu(y),$ we complete the proof.
\end{proof}

The following is the O'Neil inequality (see \cite[Lemma 1.4]{ONeil}) for the convolution operator defined by \eqref{Convolution}.
\begin{cor}\label{convolution-cor}
For the convolution bilinear mapping $\mathfrak{B}:(f,x)\to f\ast x$ on $L_1(\mathbb{T}^{d})\times L_1(\mathbb{T}^{d}_\theta)$, we have
\begin{equation*}
\int\limits_{0}^{t}\mu(s,f\ast x)ds\leq 16\Big( \int\limits_{0}^{t}\mu(s,f)ds\cdot\int\limits_{0}^{t}\mu(s,x)ds+t\int\limits_{t}^{1}\mu(s,f)\mu(s,x)ds\Big), \ t\in(0,1).
\end{equation*}
\end{cor}
\begin{proof} Let $\mathcal{M}_{1}=L_{\infty}(\mathbb{T}^{d})$ and $\mathcal{M}_{2}=\mathcal{M}_{3}=L_{\infty}(\mathbb{T}^{d}_{\theta}).$ It follows from Lemma~\ref{FirstTechLemma} and Lemma~\ref{questionable oneal lemma} that the convolution operator $ (f,x) \mapsto \frac1{16}f\ast x : L_{1}(\mathbb{T}^{d}) \times L_{1}(\mathbb{T}^{d}_{\theta}) \to L_{1}(\mathbb{T}^{d}_{\theta})$ satisfies all the assumptions in Lemma \ref{final B-lemma}. Hence, the conclusion holds.
\end{proof}

The following theorem is the main technical result of this paper.
\begin{thm}\label{main-conv-thm} Let $d\geq2$ and let $T$ be the operator defined by \eqref{operatorT}. If $f\in L_{2,\infty}(\mathbb{T}^{d})$ (see \eqref{LorentzEqvNorm}), then for any $x\in L_{1}(\mathbb{T}^{d}_{\theta})$ we have the following inequality
\begin{equation}\label{DistrForConv}
    \mu(f\ast x)\prec\prec 32\|f\|_{L_{2,\infty}(\mathbb{T}^{d})}T\mu(x).
\end{equation}
    
\end{thm}

\begin{proof} By Corollary \ref{convolution-cor} we have 
\begin{equation*}
    \int\limits_{0}^t\mu(s,f\ast x)ds\leq 16\Big( \int\limits_{0}^t\mu(s,f)ds\cdot\int\limits_{0}^t\mu(s,x)ds+t\int\limits_t^1\mu(s,f)\mu(s,x)ds\Big),\ \ t\in(0,1).
\end{equation*}
Let us estimate the terms on the right-hand side. For the first term we get
\begin{equation*}
\begin{split}
    \int\limits_{0}^{t}\mu(s,f)ds&=\int\limits_{0}^{t}s^{\frac{1}{2}}\mu(s,f)s^{-\frac{1}{2}}ds\leq \|f\|_{L_{2,\infty}(\mathbb{T}^{d})}\int\limits_{0}^{t}s^{-\frac{1}{2}}ds\\
    &=2\|f\|_{L_{2,\infty}(\mathbb{T}^{d})}\left(s^{-\frac{1}{2}+1}\Big|_{0}^{t}\right)=2\|f\|_{L_{2,\infty}(\mathbb{T}^{d})} t^{\frac{1}{2}},
    \end{split}
\end{equation*}
and for the second term we get
\begin{equation*}
\begin{split}
   t\int\limits_t^1\mu(s,f)\mu(s,x)ds&=t\int\limits_{t}^{1}s^{\frac{1}{2}}\mu(s,f)\mu(s,x)s^{-\frac{1}{2}}ds\\
   &\leq \|f\|_{L_{2,\infty}(\mathbb{T}^{d})}t\int\limits_{t}^{1}\mu(s,x)s^{-\frac{1}{2}}ds.
    \end{split}
\end{equation*}

Thus, we have 
\begin{equation}\label{SubMajor}
\begin{split}
   \int\limits_{0}^{1}\mu(s,x\ast f)ds&\leq 32\|f\|_{L_{2,\infty}(\mathbb{T}^{d})}\left(t^{\frac{1}{2}}\int\limits_{0}^{t}\mu(s,x)ds+t\int\limits_{t}^{1}\mu(s,x)s^{-\frac{1}{2}}ds\right)\\
   &=32\|f\|_{L_{2,\infty}(\mathbb{T}^{d})}F(t),
    \end{split}
\end{equation}
here $F(t)$ is defined by
\begin{equation*}
    F(t)=t^{\frac{1}{2}}\int\limits_{0}^{t}\mu(s,x)ds+t\int\limits_{t}^{1}\mu(s,x)s^{-\frac{1}{2}}ds.
\end{equation*}
Taking derivative with respect to $t>0,$ we have
\begin{equation*}
\begin{split}
    F'(t)&=\frac{1}{2}t^{-\frac{1}{2}}\int\limits_{0}^{t}\mu(s,x)ds+t^{\frac{1}{2}}\mu(t,x)+\int\limits_{t}^{1}\mu(s,x)s^{-\frac{1}{2}}ds-t^{\frac{1}{2}}\mu(t,x)\\
    &=\frac{1}{2}t^{-\frac{1}{2}}\int\limits_{0}^{t}\mu(s,x)ds+\int\limits_{t}^{1}\mu(s,x)s^{-\frac{1}{2}}ds\\
    &\leq t^{-\frac{1}{2}}\int\limits_{0}^{t}\mu(s,x)ds+\int\limits_{t}^{1}\mu(s,x)s^{-\frac{1}{2}}ds\stackrel{\eqref{operatorT}}{=} (T\mu(x))(t).
\end{split}
\end{equation*}
Hence, 
\begin{equation*}
    F'(t)\leq (T\mu(x))(t), \,\ x\in L_1(\mathbb{T}^{d}_{\theta}).
\end{equation*}
By taking the integral with respect to $t$ we get
\begin{equation*}
    F(t)\leq\int\limits_{0}^{t}(T\mu(x))(s)ds, \,\ x\in L_1(\mathbb{T}^{d}_{\theta}).
\end{equation*}
Finally, combining the last inequality and \eqref{SubMajor}, we obtain 
\begin{equation*}
    \int\limits_{0}^{t}\mu(s,x\ast f)ds\leq 32\|f\|_{L_{2,\infty}(\mathbb{T}^{d})}\int\limits_{0}^{t}(T\mu(x))(s)ds, \,\ x\in L_1(\mathbb{T}^{d}_{\theta}).
\end{equation*}
In other words, we have
\begin{equation}\label{SubMajorMain}
    \mu(f\ast x)\prec\prec 32\|f\|_{L_{2,\infty}(\mathbb{T}^{d})}T\mu(x), \ x\in L_{1}(\mathbb{T}^{d}_{\theta}).
\end{equation}
This completes the proof.
\end{proof}
\begin{lem}\label{weak_L2-function}The function $f,$ which is the inverse Fourier transform on $\mathbb{T}^{d}$ of the function
\begin{equation*}
n\mapsto (1+|n|^{2})^{-\frac{d}{4}},\ \ n\in\mathbb{Z}^{d},
\end{equation*}
belongs to $L_{2,\infty}(\mathbb{T}^{d}).$
\begin{proof} 

We identify the torus with the cube
\begin{equation*}
\mathbb{T}^{d} = \big[-\tfrac{1}{2},\tfrac{1}{2}\big]^{d}.
\end{equation*}
Since the cube $[-\tfrac12,\tfrac12]^d$ is a fundamental domain of the lattice
$\mathbb{Z}^d$ acting on $\mathbb{R}^d$, any function
$g \in L_p([-\tfrac12,\tfrac12]^d)$ admits a canonical $\mathbb{Z}^d$-periodic
extension to $\mathbb{R}^d$. We identify this extension with a function on the
torus $\mathbb{T}^d = \mathbb{R}^d / \mathbb{Z}^d$. This construction is
standard; see, for example, \cite[Chapter~3, pp. 162-163]{G2014}
or \cite[Section~2.3, pp. 153-154]{SS}.
So, the function $f$ on $\big[-\tfrac{1}{2},\tfrac{1}{2}\big]^{d}$ is the periodization of the inverse Fourier transform $\mathcal{F}^{-1}_{\mathbb{R}^{d}}$ of
$(1+|\xi|^{2})^{-\frac{d}{4}}$ on $\mathbb{R}^{d}$ (\cite[Chapter 3]{G2014}):
\begin{equation}\label{eq:periodization}
f(t)=\sum_{n\in\mathbb{Z}^{d}}
\mathcal{F}^{-1}_{\mathbb{R}^{d}}\!\left[(1+|\xi|^{2})^{-\frac{d}{4}}\right](t+n),
\quad t\in\big[-\tfrac{1}{2},\tfrac{1}{2}\big]^{d}.
\end{equation}
A standard formula (see \cite[Chapter~1, \S3]{AS}, \cite[Example~2.2.7, p.~134]{LMSZ}) yields
\begin{equation*}
\mathcal{F}^{-1}_{\mathbb{R}^{d}}\!\left[(1+|\xi|^{2})^{-\frac{d}{4}}\right](s)
= C_d\,|s|^{-\frac{d}{4}}K_{\frac{d}{4}}(|s|),\quad s\in\mathbb{R}^{d},
\end{equation*}
where $K_{\frac{d}{4}}$ is the Macdonald function. Hence,
\begin{equation}\label{eq:f-explicit-short}
\begin{split}
f(t)
&=\sum_{n\in\mathbb{Z}^{d}} C_d\,|t+n|^{-\frac{d}{4}}K_{\frac{d}{4}}(|t+n|) \\
&= C_{d}|t|^{-\frac{d}{4}}K_{\frac{d}{4}}(|t|)
   +\sum_{n\in\mathbb{Z}^{d}\setminus\{0\}} C_d\,|t+n|^{-\frac{d}{4}}K_{\frac{d}{4}}(|t+n|) \\
&=: g_{0}(t)+g_{1}(t), \ \ t\in\big[-\tfrac{1}{2},\tfrac{1}{2}\big]^{d}.
\end{split}
\end{equation}
We first show that $g_{1}$ is bounded. For $t\in[-\tfrac{1}{2},\tfrac{1}{2}]^{d}$ and $n\in\mathbb{Z}^{d}\setminus\{0\}$ there is $j$ with $|n_{j}|\ge1$ and we have
\begin{equation*}
|t+n|\geq|t_{j}+n_{j}|\geq 1-\frac{1}{2}=\frac{1}{2}.
\end{equation*}
Macdonald function satisfies (see \cite[Chapter~1, \S4]{AS}) the inequality
\begin{equation*}
|K_{\frac{d}{4}}(s)| \leq C_de^{-s},\quad s\geq\frac12.
\end{equation*}
Thus, for $t\in\big[-\tfrac{1}{2},\tfrac{1}{2}\big]^{d}$ and $n\in\mathbb{Z}^{d}\setminus\{0\}$, we have
\begin{equation*}
\big|C_d |t+n|^{-\frac{d}{4}}K_{\frac{d}{4}}(|t+n|)\big|
\leq 2^{\frac{d}{4}}C_de^{-|t+n|}\leq 2^{\frac{d}{4}}C_de^{d^{\frac12}-|n|},
\end{equation*}
and the series $\sum\limits_{n\in\mathbb{Z}^d}e^{-|n|}$ converges. Hence, $g_{1}$ is bounded on $\big[-\tfrac{1}{2},\tfrac{1}{2}\big]^{d}$.

Since adding a bounded function does not affect membership in $L_{2,\infty}\big(\mathbb{T}^d\big)$, the only possible singularity of $f$ comes from $g_{0}$. Therefore, to prove $f\in L_{2,\infty}\big(\mathbb{T}^d\big)$, it suffices to show that $g_{0}\in L_{2,\infty}\big(\big[-\tfrac{1}{2},\tfrac{1}{2}\big]^{d}\big)$. 

The Macdonald function satisfies the asymptotic expansion
(see \cite[Chapter 1, \S 4]{AS})
\begin{equation*}
K_{\nu}(r)\sim \frac12\Gamma(\nu)\Big(\frac{r}{2}\Big)^{-\nu},
\ \ r\to 0^{+},\ \nu>0.
\end{equation*}
With $\nu=\frac d4>0$, we obtain
\begin{equation*}
g_{0}(t)
\sim C_d\,|t|^{-\frac{d}{4}}\cdot
\frac{\Gamma(\frac d4)}{2^{1+\frac d4}}|t|^{-\frac{d}{4}}
= C_d'\,|t|^{-\frac{d}{2}},
\ \ |t|\to 0,
\end{equation*}
where $C_d'=C_d\cdot\frac{\Gamma(\frac d4)}{2^{1+\frac d4}}$.

Let $m$ be the Lebesgue measure on $\big[-\tfrac{1}{2},\tfrac{1}{2}\big]^{d}$. For any $\lambda>0$,
\begin{equation*}
\begin{split}
m\big(\{t\in\big[-\tfrac{1}{2},\tfrac{1}{2}\big]^{d}: |g_{0}(t)|>\lambda\}\big)
&= m\big(\{t\in\big[-\tfrac{1}{2},\tfrac{1}{2}\big]^{d}: C_d'|t|^{-d/2}>\lambda\}\big) \\
&= m\big(\{t\in\big[-\tfrac{1}{2},\tfrac{1}{2}\big]^{d}: |t|<(C_d')^{-1}\lambda^{-2/d}\}\big) \\
&=\frac{\pi^{d/2}}{\Gamma(\frac{d}{2}+1)}(C_d')^{-d}(\lambda^{-2/d})^{d}
= \widetilde C_d\,\lambda^{-2},
\end{split}
\end{equation*}
where
$ \widetilde C_d=
\frac{\pi^{d/2}}{\Gamma(\frac{d}{2}+1)}
\left(C_{d}\cdot\frac{\Gamma(\frac{d}{4})}{2^{1+\frac{d}{4}}}\right)^{-d}.$
Consequently, by \eqref{LorentzEqvNorm} we have
\begin{equation*}
\|g_{0}\|_{L_{2,\infty}\big(\big[-\tfrac{1}{2},\tfrac{1}{2}\big]^{d}\big)}
=\sup_{0<\lambda<1}\lambda\,m\big(\{|g_{0}|>\lambda\}\big)^{1/2}
=\sup_{0<\lambda<1}\lambda\,(\widetilde C_d\,\lambda^{-2})^{1/2}
=\widetilde C_d^{1/2}<\infty.
\end{equation*}
Thus $g_{0}\in L_{2,\infty}\big(\big[-\tfrac{1}{2},\tfrac{1}{2}\big]^{d}\big)$. This completes the proof.
\end{proof}
\end{lem}

Now, we can obtain our distributional Sobolev type inequality, which is the main result of this section.
\begin{thm} \label{TheoremMain}Let $d\geq2$ and let $T$ be the operator defined by \eqref{operatorT}. Let $f$ be the inverse Fourier transform on $\mathbb{T}^{d}$ of the function
\begin{equation*}
n\mapsto (1+|n|^{2})^{-\frac{d}{4}},\ \ n\in\mathbb{Z}^{d}.
\end{equation*}
Then for any $x\in L_{1}(\mathbb{T}^{d}_{\theta})$ we have
    \begin{equation}\label{d/4estimate}
    \mu((1-\Delta_{\theta})^{-\frac{d}{4}}x)\prec\prec 32\|f\|_{L_{2,\infty}(\mathbb{T}^{d})}T\mu(x).
 \end{equation}
\end{thm}
\begin{proof} Let us represent $(1-\Delta_{\theta})^{-\frac{d}{4}}$ as a convolution operator:
\begin{equation}\label{eq:convolution}
(1-\Delta_{\theta})^{-\frac{d}{4}}x = f \ast x,\ \  x \in L_{1}(\mathbb{T}^{d}_{\theta}),
\end{equation}
where $f$ is the inverse Fourier transform on $\mathbb{T}^{d}$ of the function
\begin{equation*}
n\mapsto (1+|n|^{2})^{-\frac{d}{4}},\ \ n\in\mathbb{Z}^{d}.
\end{equation*}
By Lemma \ref{weak_L2-function}, we have that $f\in L_{2,\infty}(\mathbb T^d).$ Hence, by \eqref{eq:convolution} and Theorem \ref{main-conv-thm}, we obtain 
\begin{equation*}
        \mu((1-\Delta_{\theta})^{-\frac{d}{4}}x)=\mu(f\ast x)\prec\prec 32\|f\|_{L_{2,\infty}(\mathbb{T}^{d})}T\mu(x), \ x\in L_{1}(\mathbb{T}^{d}_{\theta}),
\end{equation*}
thereby we completing the proof.
\end{proof}

\subsection{Noncommutative fractional Sobolev embeddings}

The next theorem is the main ``ideological'' result of the current section.

\begin{thm}\label{SymSobolevTh}Let $T$ be the operator defined by \eqref{operatorT}. If $E(0,1)$ is a symmetric and $F(0,1)$ is a fully symmetric function spaces such that $T$ is bounded from $E(0,1)$ to $F(0,1),$ then we have the embedding
$$H^{\frac d2}_{\mathcal{E}}(\mathbb{T}^{d}_{\theta}) \hookrightarrow \mathcal{F}(\mathbb{T}^{d}_{\theta})$$
and for any $x\in \mathcal{E}(\mathbb{T}^{d}_{\theta})$ we have
    \begin{equation*}
        \|(1-\Delta_{\theta})^{-\frac{d}{4}}x\|_{\mathcal{F}(\mathbb{T}^{d}_{\theta})}\lesssim \|x\|_{\mathcal{E}(\mathbb{T}^{d}_{\theta})}, \,\ x\in \mathcal{E}(\mathbb{T}^{d}_{\theta}).
    \end{equation*}
\end{thm}
\begin{proof} Let $F(0,1)$ be a fully symmetric space and $\mathcal{F}(\mathbb{T}^{d}_{\theta})$ be the corresponding noncommutative fully symmetric space. By Theorem \ref{TheoremMain} for any $x\in \mathcal{F}(\mathbb{T}^{d}_{\theta}),$ we have 
    \begin{equation*}
        \|(1-\Delta_{\theta})^{-\frac{d}{4}}x\|_{\mathcal{F}(\mathbb{T}^{d}_{\theta})}\stackrel{\eqref{SymNCspaceNorm}}{=}\|\mu((1-\Delta_{\theta})^{-\frac{d}{4}}x)\|_{F(0,1)}\lesssim \|T\mu(x)\|_{F(0,1)}.
    \end{equation*}
Since $T:E(0,1)\to F(0,1)$ is bounded by the assumption, it follows that
 \begin{equation}\label{xinTh4.10}
     \|(1-\Delta_{\theta})^{-\frac{d}{4}}x\|_{\mathcal{F}(\mathbb{T}^{d}_{\theta})}\lesssim\|T\mu(x)\|_{F(0,1)}\lesssim \|\mu(x)\|_{E(0,1)}\stackrel{\eqref{SymNCspaceNorm}}{=}\|x\|_{\mathcal{E}(\mathbb{T}^{d}_{\theta})}, \,\ x\in \mathcal{E}(\mathbb{T}^{d}_{\theta}).
 \end{equation}
Now, for any $y\in H^{\frac{d}{2}}_{\mathcal{E}}(\mathbb{T}^{d}_{\theta})$, we obtain
 \begin{equation*}
 \begin{split}
     \|y\|_{\mathcal{F}(\mathbb{T}^{d}_{\theta})}&=\|(1-\Delta_{\theta})^{-\frac{d}{4}}(1-\Delta_{\theta})^{\frac{d}{4}}y\|_{\mathcal{F}(\mathbb{T}^{d}_{\theta})}\\
     &\stackrel{\eqref{xinTh4.10}}{\lesssim}\|(1-\Delta_{\theta})^{\frac{d}{4}}y\|_{\mathcal{E}(\mathbb{T}^{d}_{\theta})}=\|y\|_{H_{\mathcal{E}}^{\frac{d}{2}}(\mathbb{T}^{d}_{\theta})}, \,\ y\in H^{\frac{d}{2}}_{\mathcal{E}}(\mathbb{T}^{d}_{\theta}).
 \end{split}
 \end{equation*}
 In other words, we have the following embedding
 \begin{equation*}
     H^{\frac d2}_{\mathcal{E}}(\mathbb{T}^{d}_{\theta}) \hookrightarrow \mathcal{F}(\mathbb{T}^{d}_{\theta}),
 \end{equation*}
thereby completing the proof.
\end{proof}

\begin{cor}\label{ThmSEmb} Let $d\geq2$ and $1\leq p\leq\infty$. Let $\phi$ be the function defined by \eqref{PhiPsiFunc}. Then we have the following noncommutative fractional Sobolev inequalities 
    \begin{itemize}
         \item [(i)] $\|(1-\Delta_{\theta})^{-\frac{d}{4}}x\|_{L_{p}(\mathbb{T}^{d}_{\theta})}\lesssim \|x\|_{L_{p}(\mathbb{T}^{d}_{\theta})}, \ x\in L_{p}(\mathbb{T}^{d}_{\theta})$;
        \item [(ii)] $\|(1-\Delta_{\theta})^{-\frac{d}{4}}x\|_{\Lambda^{(2)}_{\phi}(\mathbb{T}^{d}_{\theta})}\lesssim \|x\|_{L_{2}(\mathbb{T}^{d}_{\theta})}, \ x\in L_{2}(\mathbb{T}^{d}_{\theta}).$
           \end{itemize}
\end{cor}
\begin{proof}
(i) It follows from Proposition \ref{ThmTBdd} (i) that $T$ is bounded on $L_{p}(0,1), 1\leq p\leq\infty$, and it is known that the $L_{p}$-spaces ($1\leq p\leq\infty$) are fully symmetric (see, \cite[Chapter 6, Theorem 6.2.1, p. 399]{DdPS}). Then, by taking $F(0,1)=E(0,1)=L_{p}(0,1)$ in Theorem \ref{SymSobolevTh}, we obtain
\begin{equation*}
    \|(1-\Delta_{\theta})^{\frac{d}{4}}x\|_{L_{p}(\mathbb{T}^{d}_{\theta})}\lesssim\|x\|_{L_{p}(\mathbb{T}^{d}_{\theta})}, \ x\in L_{p}(\mathbb{T}^{d}_{\theta}).
\end{equation*}

(ii) Once again, Proposition \ref{ThmTBdd} (ii) yields that  $T$ is bounded from $L_{2}(0,1)$ to $\Lambda^{(2)}_{\phi}(0,1)$, and recall that, the space $\Lambda^{(2)}_{\phi}(0,1)$ is fully symmetric \cite[Chapter 6, Theorem 6.3.8, p. 405]{DdPS}. By taking $E(0,1)=L_{2}(0,1)$ and $F(0,1)=\Lambda^{(2)}_{\phi}(0,1)$ in Theorem \ref{SymSobolevTh}, we conclude
\begin{equation*}
   \|(1-\Delta_{\theta})^{-\frac{d}{4}}x\|_{\Lambda^{(2)}_{\phi}(\mathbb{T}^{d}_{\theta})}\lesssim \|x\|_{L_{2}(\mathbb{T}^{d}_{\theta})}, \ x\in L_{2}(\mathbb{T}^{d}_{\theta}).
   \end{equation*}
   This completes the proof.
\end{proof}
\begin{rem}Note that Theorem~\ref{SymSobolevTh}, together with Corollary~\ref{ThmSEmb} (i), yields the result of \cite[Theorem~6.6 (i)]{XXY}, while combined with Corollary~\ref{ThmSEmb} (ii) it gives a noncommutative version of the Hansson–Brezis–Wainger inequality proved in \cite[Proposition~7]{SZ}. In the special case when $\theta=0$, our results in Corollary \ref{ThmSEmb} recover numerous classical results on Sobolev embedding due to Hansson \cite{Hansson}, Brezis and Wainger \cite{BW}, and Cwikel and Pustylnik \cite{CP}.
\end{rem}

Now, let $x\in L_{\infty}(\mathbb{T}^{d}_{\theta})$, and $M_{x}:L_{2}(\mathbb{T}^{d}_{\theta})\to L_{2}(\mathbb{T}^{d}_{\theta})$ be a multiplication operator given by $M_{x}(y)=yx$. The next theorem is a noncommutative analogue of the Cwikel--Solomyak type estimate (see, \cite[Proposition 19]{SZ}).

\begin{thm}\label{C-SforM}Let $d\geq2$. For any $x\in \mathtt{M}_{\phi}(\mathbb{T}^{d}_{\theta})$ we have
    \begin{equation*}
        \|(1-\Delta_{\theta})^{-\frac{d}{4}}M_{x}(1-\Delta_{\theta})^{-\frac{d}{4}}\|_{L_{2}(\mathbb{T}^{d}_{\theta})\to L_{2}(\mathbb{T}^{d}_{\theta})}\lesssim\|x\|_{\mathtt{M}_{\phi}(\mathbb{T}^{d}_{\theta})}.
    \end{equation*}
\end{thm}
\begin{proof}  Without loss of generality, we may assume that $x\in \mathtt{M}_{\phi}(\mathbb{T}^{d}_{\theta})$ is positive. Recall that, the K\"othe dual of $\Lambda_{\phi}$ is the Marcinkiewicz space $\mathtt{M}_{\phi}$ (see, \cite[Theorem 6.4.9, p. 412]{DdPS}), and for a positive operator $K:L_{2}(\mathbb{T}^{d}_{\theta})\to L_{2}(\mathbb{T}^{d}_{\theta})$, we have
\begin{equation*}
    \|K\|_{L_{2}(\mathbb{T}^{d}_{\theta})\to L_{2}(\mathbb{T}^{d}_{\theta})}=\sup\limits_{\|y\|_{L_{2}(\mathbb{T}^{d}_{\theta})}\leq1}|\langle Ky, y\rangle|=\sup\limits_{\|y\|_{L_{2}(\mathbb{T}^{d}_{\theta})}\leq1}|\tau_{\theta}( Ky\cdot y^{*})|, \ y\in L_{2}(\mathbb{T}^{d}_{\theta}).
\end{equation*}
On the other hand, for $x\in\mathtt{M}_{\phi}(\mathbb{T}^{d}_{\theta})$ and $ y\in L_{2}(\mathbb{T}^{d}_{\theta})$ we have
\begin{equation*}
    \left\langle(1-\Delta_{\theta})^{-\frac{d}{4}}M_{x}(1-\Delta_{\theta})^{-\frac{d}{4}}y,y\right\rangle=\tau_{\theta}\left(x\left|((1-\Delta_{\theta})^{-\frac{d}{4}}y)^{*}\right|^{2}\right).
\end{equation*}
Then, we obtain 
\begin{equation*}
\begin{split}
    \|(1-\Delta_{\theta})^{-\frac{d}{4}}M_{x}(1-\Delta_{\theta})^{-\frac{d}{4}}\|_{L_{2}(\mathbb{T}^{d}_{\theta})\to L_{2}(\mathbb{T}^{d}_{\theta})}&=\sup\limits_{\|y\|_{L_{2}(\mathbb{T}^{d}_{\theta})}\leq1}\left|\left\langle (1-\Delta_{\theta})^{-\frac{d}{4}}M_{x}(1-\Delta_{\theta})^{-\frac{d}{4}}y,y \right\rangle\right|\\
    &=\sup\limits_{\|y\|_{L_{2}(\mathbb{T}^{d}_{\theta})}\leq1}\left| \tau_{\theta}\left(x \cdot \left|\left((1-\Delta_{\theta})^{-\frac{d}{4}}y\right)^{*}\right|^{2}\right)\right|.
    \end{split}
    \end{equation*}
By \cite[Formula 4.12]{DdPS} we have
\begin{equation}
\|xy\|_{L_{1}(\mathbb{T}^{d}_{\theta})}\leq \|x\|_{\mathtt{M}_{\phi}(\mathbb{T}^{d}_{\theta})}\|y\|_{\Lambda_{\phi}(\mathbb{T}^{d}_{\theta})}, \ x\in \mathtt{M}_{\phi}(\mathbb{T}^{d}_{\theta}), \ y\in \Lambda_{\phi}(\mathbb{T}^{d}_{\theta}).
\end{equation}
Applying this formula and by definition of the 2-convexification \eqref{2ConvNorm}, we obtain
    \begin{equation*}
        \begin{split}
                \|(1-\Delta_{\theta})^{-\frac{d}{4}}M_{x}(1-\Delta_{\theta})^{-\frac{d}{4}}\|_{L_{2}(\mathbb{T}^{d}_{\theta})\to L_{2}(\mathbb{T}^{d}_{\theta})} 
                &\leq\sup\limits_{\|y\|_{L_{2}(\mathbb{T}^{d}_{\theta})}\leq1}\left\| x \cdot\left|(1-\Delta_{\theta})^{-\frac{d}{4}}y^{*}\right|^{2} \right\|_{L_{1}(\mathbb{T}^{d}_{\theta})}\\
                &\leq \sup\limits_{\|y\|_{L_{2}(\mathbb{T}^{d}_{\theta})}\leq1} \|x\|_{\mathtt{M}_{\phi}(\mathbb{T}^{d}_{\theta})}\left\|\left((1-\Delta_{\theta})^{-\frac{d}{4}}y^{*}\right)^{2} \right\|_{\Lambda_{\phi}(\mathbb{T}^{d}_{\theta})} \\
                &= \sup\limits_{\|y\|_{L_{2}(\mathbb{T}^{d}_{\theta})}\leq1} \|x\|_{\mathtt{M}_{\phi}(\mathbb{T}^{d}_{\theta})}\left\|(1-\Delta_{\theta})^{-\frac{d}{4}}y^{*} \right\|^{2}_{\Lambda_{\phi}^{(2)}(\mathbb{T}^{d}_{\theta})}.
\end{split}
\end{equation*}
Hence, for $x\in\mathtt{M}_{\phi}(\mathbb{T}^{d}_{\theta})$ and $y\in L_{2}(\mathbb{T}^{d}_{\theta}),$  by Corollary \ref{ThmSEmb} (ii) we have
\begin{equation*}
        \|(1-\Delta_{\theta})^{-\frac{d}{4}}M_{x}(1-\Delta_{\theta})^{-\frac{d}{4}}\|_{L_{2}(\mathbb{T}^{d}_{\theta})\to L_{2}(\mathbb{T}^{d}_{\theta})}\lesssim\sup\limits_{\|y\|_{L_{2}(\mathbb{T}^{d}_{\theta})}\leq1} \|x\|_{\mathtt{M}_{\phi}(\mathbb{T}^{d}_{\theta})} \|y^{*}\|^{2}_{L_{2}(\mathbb{T}^{d}_{\theta})}\lesssim \|x\|_{\mathtt{M}_{\phi}(\mathbb{T}^{d}_{\theta})},
\end{equation*}
thereby completing the proof.
\end{proof}

\subsection{$L_{2}$-time decay for the heat type equation} Let us consider the following heat-type equation for $s>0$
\begin{equation}\label{Main-equation-1}
\partial_{t}^{\alpha}(u(t)-u_{0}) + (1-\Delta_{\theta})^{\frac{s}{2}}u(t) = 0, \ \ t > 0, \ \ 0 < \alpha \leq 1, 
\end{equation}
with the initial data
\begin{equation}\label{Initial-date-1}
u(0) = u_{0}\in L_{2}(\mathbb{T}^{d}_{\theta}).   
\end{equation} 

\begin{definition} Let $s>0$ and $u_{0}\in L_{2} (\mathbb{T}^{d}_{\theta}).$ Then the operator $u$ defined by the formula
\begin{equation*}
    u(t)=E_{\alpha}(-t^{\alpha}(1-\Delta_{\theta})^{\frac{s}{2}}u_0):=\sum\limits_{n\in\mathbb{Z}^{d}}E_{\alpha}(-t^{\alpha}(1+|n|^{2})^{\frac{s}{2}})\widehat{u}_{0}(n)e^{\theta}_{n}, \ t>0,
\end{equation*} 
is called a mild solution to the problem \eqref{Main-equation-1}-\eqref{Initial-date-1}. When $\theta=0$, the same definition is used in \cite{VergaraPresentAuthors, VergaraZacher2010}.

\end{definition}

Similar questions were investigated in \cite{STT}, where the authors obtained $L_{p}$-$L_{q}$ time decay estimates for $1<p\leq 2\leq q<\infty$.  In the particular case $p=q=2$, however, their approach does not yield an $L_{2}$-decay estimate for $t>0$, since the method loses any $t$--dependence and provides only a constant bound instead of a genuine decay.
By using Theorem \ref{TheoremMain}, we obtain the following theorem, which complements results in \cite[Theorem 4.1.2, (i)-(ii)]{STT} for the case $p=q=2$.

\begin{thm}\label{L2timeest}Let $s=\frac{d}{2}, \ \alpha\in(0,1]$ and $0\neq u_{0} \in L_{2}(\mathbb{T}^{d}_{\theta}).$ Assume that $u$ is a mild solution to the Cauchy problem \eqref{Main-equation-1}-\eqref{Initial-date-1}. Then we have the following $L_{2}$-time decay for the mild solution
\begin{equation*}
    \|u(t)\|_{L_{2}(\mathbb{T}^{d}_{\theta})}\lesssim t^{-\alpha},\,\ t>0.
\end{equation*}
\end{thm}
\begin{proof} Let us take the $L_{2}$-norm from the mild solution $u$:
\begin{equation*}
\begin{split}
    \|u(t)\|^{2}_{L_{2}(\mathbb{T}^{d}_{\theta})}&=\|\widehat{u}(t)\|^{2}_{\ell_{2}(\mathbb{Z}^{d})}=\sum\limits_{n\in\mathbb{Z}^{d}}|E_{\alpha}(-t^{\alpha}(1+|n|^{2})^{\frac{d}{4}})\widehat{u}_{0}(n)|^{2}\\
    & \overset{\eqref{ML-estimate}}{\leq} \sum\limits_{n\in\mathbb{Z}^{d}}\frac{c_{1}}{(1+t^{\alpha}(1+|n|^{2})^{\frac{d}{4}})^{2}}|\widehat{u}_{0}(n)|^{2}\\
    &\lesssim t^{-2\alpha} \sum\limits_{n\in\mathbb{Z}^{d}}\frac{t^{2\alpha}(1+|n|^{2})^{\frac{d}{2}}}{(1+t^{\alpha}(1+|n|^{2})^{\frac{d}{4}})^{2}}(1+|n|^{2})^{-\frac{d}{2}}|\widehat{u}_{0}(n)|^{2}\\
    &\lesssim t^{-2\alpha} \sum\limits_{n\in\mathbb{Z}^{d}}|(1+|n|^{2})^{-\frac{d}{4}}\widehat{u}_{0}(n)|^{2}=t^{-2\alpha}\|(1+|\cdot|^{2})^{-\frac{d}{4}}\widehat{u}_{0}(\cdot)\|^{2}_{\ell_{2}(\mathbb{Z}^{d})}\\
    &\overset{\eqref{Parceveal}}{=}t^{-2\alpha}\|(1-\Delta_{\theta})^{-\frac{d}{4}}(u_{0})\|^{2}_{L_{2}(\mathbb{T}^{d}_{\theta})}. 
    \end{split}
\end{equation*}

Note that, $(1-\Delta_{\theta})^{-\frac{d}{4}}$ is the convolution operator such that $(1-\Delta_{\theta})^{-\frac{d}{4}}(u_0)=f*u_0$ (see formula \eqref{eq:convolution}), where $f$ is the inverse Fourier transform on $\mathbb{T}^{d}$ of the function
\begin{equation*}
n\mapsto (1+|n|^{2})^{-\frac{d}{4}},\ \ n\in\mathbb{Z}^{d}.
\end{equation*}
It was shown in Therorem \ref{TheoremMain} that $f\in L_{2,\infty}(\mathbb{T}^{d}).$
Hence, since $u_{0}\in L_{2}(\mathbb{T}^{d}_{\theta}),$ applying Theorem \ref{TheoremMain} and Proposition \ref{ThmTBdd} (i) (for $p=2$), we obtain 
\begin{equation*}
    \begin{split}
            \|u(t)\|_{L_{2}(\mathbb{T}^{d}_{\theta})}&\lesssim t^{-\alpha}\|\mu((1-\Delta_{\theta})^{-\frac{d}{4}}(u_{0}))\|_{L_{2}(0,1)} \overset{\eqref{d/4estimate}}{\leq} 32\,t^{-\alpha}\|f\|_{L_{2,\infty}(\mathbb{T}^{d})}\|T\mu(u_{0})\|_{L_{2}(0,1)}\\
    &\leq 32\,t^{-\alpha}\|f\|_{L_{2,\infty}(\mathbb{T}^{d})}\|\mu(u_{0})\|_{L_{2}(0,1)}=32\,t^{-\alpha}\|f\|_{L_{2,\infty}(\mathbb{T}^{d})}\|u_{0}\|_{L_{2}(\mathbb{T}^{d}_{\theta})}.
\end{split}
\end{equation*}
By assumption $\|u_{0}\|_{L_{2}(\mathbb{T}^{d}_{\theta})}$ is finite. Hence, we get
\begin{equation*}
    \|u(t)\|_{L_{2}(\mathbb{T}^{d}_{\theta})}\lesssim t^{-\alpha},\,\ t>0.
\end{equation*}
This completes the proof.
\end{proof}

\section{Acknowledgements}

The work was partially supported by the Australian Research Council. K.T. was partially supported by ARC grant FL17010005. F.S. was supported by ARC grant FL17010005 and ARC grant DP230100434. D.Z. was supported by ARC grant DP230100434. 


\begin{center}

\end{center}

\begin{thebibliography}{999} 

 
\bibitem{AS} N. Aronszajn, K. Smith, {\it Theory of Bessel potentials}. I. Ann. Inst. Fourier {\bf 11} (1961), 385--475.

\bibitem{BS}  C. Bennett, R. Sharpley, {\it Interpolation of operators, Pure and Applied Mathematics}, vol. 129, Academic Press, Inc., Boston, MA, 1988. MR928802

\bibitem{BW} H. Brezis, S. Wainger. {\it A note on limiting cases of Sobolev embeddings and convolution inequalities.} Commun. Partial Differ. Equ. 5(7), 773–789 (1980)

\bibitem{Chasseigne2006} E.~Chasseigne, M.~Chaves, J.D.~Rossi, {\it Asymptotic behavior for nonlocal diffusion equations,} J. Math. Pures Appl. 86 (3) (2006) 271--291.
 
\bibitem{CXY} Z. Chen, Q. Xu, Z. Yin. {\it Harmonic analysis on quantum tori}. Comm. Math. Phys. {\bf 322}:3 (2013), 755--805.

\bibitem{CPS} A. Cianchi, L. Pick, L. Slavíková. {\it  Sobolev embeddings in Orlicz and Lorentz spaces with measures.} J. Math. Anal. Appl. 485 (2020), no. 2, 123827, 31 pp.

\bibitem{Cianchi} A. Cianchi. {\it Symmetrization and second-order Sobolev inequalities}. Ann. Mat. Pura Appl. 183 (2004), 45-77.

\bibitem{Connes1980} A. Connes. {\it $C^*$-alg\`ebres et g\'eom\'etrie diff\'erentielle}.  C. R. Acad. Sc. Paris, s\'er. A, {\bf 290} (1980), 599--604.

\bibitem{CM2014} A. Connes, H. Moscovici, {\it Modular curvature for noncommutative two-tori}.   J. Amer. Math. Soc. {\bf 27} (2014), 639--684.

\bibitem{CP} M. Cwikel, E. Pustylnik. {\it Sobolev type embeddings in the limiting case.} J. Fourier Anal. Appl. 4(4–5), 433–446 (1998)

\bibitem{CzerKamin} M. M. Czerwińska, A. H. Kaminska, {\it Geometric properties of noncommutative symmetric spaces of measurable operators and unitary matrix ideals.} Commentationes Mathematicae. 57 (2017), no. 1, 45-122.

\bibitem{DdPS}  P.G. Dodds, B. de Pagter and F.A. Sukochev. \emph{Noncommutative Integration and Operator Theory}. Springer, Berlin, Heidelberg. 2023.

\bibitem{DT} T. K. Donaldson, N. S. Trudinger. \emph{Orlicz-Sobolev Spaces and Imbedding Theorems.} J. Funct. Anal., 1971.

\bibitem{EKP} D. E. Edmunds, R. Kerman and L. Pick. {\it Optimal Sobolev Imbeddings Involving Rearrangement-Invariant Quasinorms.} J. Funct. Anal. 170 (2000), no. 2, 307–355.

\bibitem{Evans} L. C. Evans. {\it Partial Differential Equations}, 2nd ed, American Math Society, 2010.


\bibitem{GO} A. Gogatishvili, V. I. Ovchinnikov. {\it Interpolation orbits and optimal Sobolev's embeddings},  J. Funct. Anal., 253(1) (2007) 1-17. 

\bibitem{G2014} L.~Grafakos. {\it Classical Fourier Analysis}. Third Edition. Springer-Verlag, New York (2014)

\bibitem{HaLeePonge} H. Ha, G. Lee, R. Ponge. {\it Pseudodifferential calculus on noncommutative tori, I. Oscillating integrals.} Internat. J. Math. {\bf 30} (2019), no. 8, 1950033, 74 pp.

\bibitem{Hansson} K. Hansson. {\it Imbedding theorems of Sobolev type in potential theory.} Math. Scand. 45(1), 77–102 (1979)

\bibitem{IgnatRossi2010} L.I.~Ignat, J.D.~Rossi. {\it Decay estimates for nonlocal problems via energy methods,} J. Math. Pures Appl. 92 (2) (2009) 163–187.


\bibitem{VergaraPresentAuthors} J.~Kemppainen, J.~Siljander, V.~Vergara, R.~Zacher. {\it Decay estimates for time-fractional and other non-local in time subdiffusion equations in $\mathbb{R}^d$}, Math. Ann. 366 (3) (2016) 941–979.

\bibitem{KSZ}  J.~Kemppainen, J.~Siljander and  R.~Zacher. {\it Representation of solutions and large-time behavior for fully nonlocal diffusion equations.} J. Differ. Equ. 263 (2017) 149-201.

\bibitem{KS} N. Kalton, F. Sukochev. {\it Symmetric norms and spaces of operators}, J. Reine Angew. Math. {\bf 621} (2008), 81--121.

\bibitem{Klimov1969} V. S. Klimov. {\it Imbedding theorems for symmetric spaces.} Math. USSR-Sb., (1969), Volume 8, Issue 2, Pages 161–168.

\bibitem{Klimov1970} V. S. Klimov. {\it On imbedding theorems for symmetric spaces.} Math. USSR-Sb., 11:3 (1970), 339–353.

\bibitem{Klimov1972} V. S. Klimov. {\it Imbedding theorems for Orlicz spaces and their application to boundary value problems}, Siberian Math. J., 13:2 (1972), 231–240.

\bibitem{KPS} S. G. Krein, Ju. I. Petunin, E. M. Semenov. {\it Interpolation of Linear Operators}. Translations of Mathematical
Monographs, vol. 54. American Mathematical Society, Providence (1982)

\bibitem{L} L. Lafleche. \emph{On Quantum Sobolev Inequalities.} J. Funct. Anal., 286(10):Paper No. 110400, 40, 2024

\bibitem{LT} J. Lindenstrauss, L. Tzafriri. Classical Banach spaces II, Springer-Verlag, Berlin, 1979.

\bibitem{LMSZ} S.~Lord, E.~McDonald, F.~Sukochev, D.~Zanin.  {\it Singular Traces: Volume 2 Trace Formulas}, Berlin, Boston: De Gruyter, (2023) 

\bibitem{LSZ} S.~Lord, F.~Sukochev, D.~Zanin.  {\it Singular Traces: Volume 1 Theory}, Berlin, Boston: De Gruyter, (2021) 

\bibitem{McDSX} E.~McDonald, F.~Sukochev, X.~Xiong. {\it Quantum differentiability on quantum tori}.  Commun. Math. Phys. 371,  1231-1260 (2019). 

\bibitem{ONeil} R. O’Neil. {\it Convolution operators and $L(p,q)$ spaces.} Duke Math. J. 30, 129–142 (1963)

\bibitem{PXu} G.~Pisier, Q.~Xu. {\it  Noncommutative $L_{p}$-spaces. Handbook of the geometry of Banach spaces}. Vol. 2, pages 1459–1517. North-Holland, Amsterdam, 2003.

\bibitem{Podlubny} I. Podlubny. \emph{Fractional Differential Equations}. Academic Press, New York, 1999.

\bibitem{Reference38} J. D. Rossi. {\it Asymptotics for evolution problems with nonlocal diffusion,} manuscript available at http://mate.dm.uba.ar/~jrossi/CURSO(Marra)25-3-08.pdf, 2009.

\bibitem{RST} M. Ruzhansky, S. Shaimardan, K. Tulenov. \emph{H\"ormander type Fourier multiplier theorem and Nikolskii inequality on quantum tori, and applications}. J. Fourier Anal. Appl. {\bf 32}, 7 (2026). https://doi.org/10.1007/s00041-025-10217-z

\bibitem{SchmeisserTriebel} H.J. Schmeisser, H. Triebel, {\it Topics in Fourier Analysis and Function Spaces,} Wiley-Interscience, Chichester, 1987.

\bibitem{STT} S.~Shaimardan, R.A.~Tastankul, and K.S.~Tulenov. \emph{Fourier multiplier on noncommutative torus and its applications to nonlinear equations}, Math. Z. 312, 16 (2026). https://doi.org/10.1007/s00209-025-03899-0

\bibitem{STo} A. G. Smadiyeva, B. T. Torebek, Decay estimates for the time-fractional evolution equations with time-dependent coefficients, \emph{Proceedings of the Royal Society A: Mathematical, Physical and Engineering Sciences},  479:2276 (2023), 20230103

\bibitem{Sobolev} S.L. Sobolev.{\it On a theorem of functional analysis.} Mat. Sb. 4 (46) 1938, 39-68.

\bibitem{Spera} M. Spera.  {\it Sobolev theory for noncommutative tori.} Rend. Sem. Mat. Univ. Padova {\bf 86} (1992), 143--156.

\bibitem{SS} E. M. Stein, R, Shakarchi. {\it Fourier Analysis: An Introduction, Princeton Lectures in Analysis}, Vol. 1, Princeton University Press, 2003.


\bibitem{F} F. Sukochev,. {\it Completeness of quasi-normed symmetric operator spaces}, Indag. Math. (N.S.) {\bf 25}:2 (2014), 376--388.

\bibitem{SZ} F. Sukochev, D. Zanin. {\it Optimal Cwikel–Solomyak Estimates}. J. Fourier Anal. Appl. 29, 21 (2023)

\bibitem{STZ}  F. Sukochev, K. Tulenov, D. Zanin. {\it Sobolev projection on quantum torus, its complete boundedness and applications}, J. Math. Anal. Appl. 543 (2025), no. 2, Paper No. 128906, 21.

\bibitem{T1} N.S. Trudinger. \emph{On Imbeddings into Orlicz Spaces and Some Applications.} J. Math. Mech., 1967.

\bibitem{Vazquez2007} J.~L.~Vázquez. {\it Barenblatt solutions and asymptotic behaviour for a nonlinear fractional heat equation of porous medium type}, J. Eur. Math. Soc. (JEMS) 16 (4) (2014) 769–803.
   
\bibitem{VergaraZacher2010} V.~Vergara, R.~Zacher, {\it Optimal decay estimates for time-fractional and other nonlocal sub diffusion equations via energy methods}, SIAM J. Math. Anal. 47 (1) (2015) 210–239.
 

\bibitem{XXY} X. Xiong, Q. Xu,  Z. Yin.    {\it Sobolev, Besov and Triebel-Lizorkin spaces on quantum tori}. Mem. Amer. Math. Soc. {\bf 252}:1203 (2018), 86 pages.

\end{thebibliography}
\end{document}